\documentclass{article}

\usepackage[utf8]{inputenc}
\usepackage[T1]{fontenc}
\usepackage{graphicx}
\usepackage{tikz}
\usepackage{amsmath,mathtools,amsfonts,amsthm}
\usepackage{siunitx}
\usepackage[pdftex,colorlinks=true,linkcolor=blue,citecolor=blue,urlcolor=blue]{hyperref}

\newcommand\R{\mathbb{R}}
\newcommand\N{\mathbb{N}}
\newcommand\D{\mathrm{d}}
\newcommand\I{\mathcal{I}_N}
\renewcommand\P{\mathcal{P}_N}
\renewcommand\S{\mathfrak{S}_N}
\newcommand\s{\sigma}
\newcommand\ts{\widetilde{\s}}

\newtheorem{remark}{Remark}
\newtheorem{lemma}{Lemma}

\newtheorem{theorem}{Theorem}

\title{Optimal periodic resource allocation in reactive dynamical systems: application to Microalgal production}

\author{Olivier Bernard$^{1,2}$, Liu-Di LU$^{1,3,4}$, Julien Salomon$^{3,4}$}
\date
{%
\noindent{\small\textit{$^1$INRIA Sophia Antipolis Méditerranée, BIOCORE Project-Team, Université Nice Côte d'Azur, 2004, Route des Lucioles - BP 93, 06902 Sophia-Antipolis Cedex, France}}\\
{\small\textit{$^2$Sorbonne Université, INSU-CNRS, Laboratoire d'Océanographie de Villefranche, 181 Chemin du Lazaret, 06230 Villefranche-sur-mer, France}}\\
{\small\textit{$^3$INRIA Paris, ANGE Project-Team, 75589 Paris Cedex 12, France}}\\
{\small\textit{$^4$Sorbonne Université, CNRS, Laboratoire Jacques-Louis Lions, 75005 Paris, France}}\\%
}

\begin{document}

\maketitle

\begin{abstract} 
In this paper we focus on a periodic resource allocation problem applied on a dynamical system which comes from a biological system. 
More precisely, we consider a system with $N$ resources and $N$ activities, each activity use the allocated resource to evolve up to a given time $T>0$ where a control (represented by a given permutation) will be applied on the system to re-allocate the resources. 
The goal is to find the optimal control strategies which optimize the cost or the benefit of the system. 
This problem can be illustrated by an industrial biological application, namely the optimization of a mixing strategy to enhance the growth rate in a microalgal raceway system.  
A mixing device, such as a paddle wheel, is considered to control the rearrangement of the depth of the algae cultures hence the light perceived at each lap. 
We prove that if the dynamics of the system is periodic, then the period corresponds to one re-allocation whatever the order of the involved permutation matrix is. 
A nonlinear optimization problem for one re-allocation process is then introduced. 
Since $N!$ permutations need to be tested in the general case, it can be numerically solved only for a limited number of $N$. 
To overcome this difficulty, we introduce a second optimization problem which provides a suboptimal solution of the initial problem, but whose solution can be determined explicitly. 
A sufficient condition to characterize cases where the two problems have the same solution is given. 
Some numerical experiments are performed to assess the benefit of optimal strategies in various settings.
\end{abstract}

\begin{paragraph}{Keywords: }
Resource Allocation, Nonlinear Problems, Periodic Control, Dynamical System, Microalgae production, Periodic System, Impulse Control, Switched Systems, Permutation, Linear Approximation, Assignment Problem.
\end{paragraph}

\section{Introduction}

Considering a fixed amount of the resources and a set of activities, we look for a distribution strategy which optimizes a given objective function.
This is the so-called resource allocation problem~\cite{Ibaraki1988}.
Due to its simple structure, this problem is encountered in a number of applications including load scheduling~\cite{Zhang201102}, manufacturing~\cite{Zhang201108}, portfolio selection~\cite{Hennessy2002} and computational biological problem~\cite{Akhmetzhanov2011}.
Periodic versions have also been considered. 
The periodic scheduling problem was first addressed in~\cite{Liu1973} the framework of operation research. 
Later on, the concept of \emph{proportionate fairness} constraint has been introduced~\cite{Baruah1996} to design allocation algorithms which schedule the resources in proportion to task weight.
Periodic resource allocation problems are also used in ecology, e.g., in~\cite{Cominetti2009,Piazza2010} where the authors investigate long-term behaviour of harvesting policies for a forest composed of multiple species with different maturity ages. 
In such problems, the state can also be described in terms of dynamical systems. 
As an example, hospital resources (hospital beds) continuous allocation is studied in~\cite{Abdelrazec2016} as a strategy to control the dengue fever, associated with a patient recovery rate. 
In the same way, a population of a single species with logistic growth in a patchy environment is considered in~\cite{Nagahara2021}.  
The problem here consists of the maximization of the total population by re-distributing the limited resources among the patches.

In general, resource allocation problems are related to the assignment of a resource to a sequence of two or more tasks. 
However, we focus in this paper on problems where $N$ resources are assigned to $N$ tasks. %
Additionally, we consider permanent regimes which are often relevant in the case of long term processes, as, e.g., crop harvesting, scheduling of appliances, etc. 
Moreover, here we also account for the dynamical evolution of the system between two re-allocations, further increasing the difficulty of the analysis.
In this way, our work is related to the fields of switched systems~\cite{Liberzon1999}, impulse control~\cite{Bensoussan1982,Hespanha2002} and to periodic control~\cite{Colonius1988}. 
These techniques are usually used to tackle stabilization issues. 
In this paper, we consider them in view of optimization issues.

In order to model the allocation process in the periodic system, we study the following allocation problem :  
\textit{Consider a system with $N$ resources and $N$ activities, each activity uses the allocated resource to evolve during a given time $T>0$.
At time $T$, an extra control is applied to re-allocate the resources according to a given permutation.}
It is proven that if the dynamics of the system is periodic, then it is one period corresponding to one allocation process whatever the order of the considered control strategy is.
A nonlinear problem is then introduced in order to find the optimal control strategies.
Since $N!$ permutations need to be tested in the general case, it can be numerically solved only for a limited number of $N$.
To overcome this difficulty, we propose a second optimization problem - a typical assignment problem - associated with a suboptimal solution of the initial problem for which its optimal control can be determined explicitly. 
In addition, a sufficient condition is provided to characterize cases when the two problems have the same solution.

For the sake of concreteness, we illustrate our theory by an industrial biological application, namely the mixing of microalgae cells in a cultivation set-up.
This emerging application has a promising potential, ranging from food to renewable energy~\cite{Wijffels2010,Hu2008,Scott2010} and is also involved in many high added value commercial applications such as pharmaceutical processes, cosmetics or pigments~\cite{Eppink2017,Spolaore2006}.
Outdoor algal cultivation is mainly carried out in open raceway ponds exposed to solar radiation. 
This hydrodynamic system is set in motion by a paddle wheel which homogenizes the medium for ensuring an equidistribution of the nutrients and guarantees that each cell will have regularly access to the light~\cite{Chiaramonti2013}. 
Microalgae then grow between two re-distributions depending on the light intensity received in their layer.
Different strategies have been proposed to optimize the production of the biomass in this algal raceway system~\cite{DelaHozSiegler2012,Cuaresma2011,Cornet2009,Cornet2010,Amos2004,Bernard2021b}. 
First studies about the mixing policy have shown that a well-chosen mixing strategy may improve the algal growth~\cite{Bernard2020,Bernard2021c}.
These works focus on algal production in a non-flat raceway system and assume constant velocity of the fluid and periodicity of the photosynthetic activity. 
The influence of the mixing strategy on the algal productivity is investigated only numerically by identifying the paddle wheel as a mixing device and modeling it by permutation~\cite{Bernard2021a}.
Finally an approximation of the functional to optimize which gives rise to an explicit solution, whereas the original problem can only be solved at high computational cost.

In the current study, we extend these preliminary works to a general class of resource allocation problems, identify the periodic solution of the underlying dynamical system as an asymptotic steady state and develop a complete theory of the proposed approximation. 
In particular, our analysis enables us to establish a criterion to compare the solutions of  the original and approximate resource allocation problems. 
New numerical results complete this study.

The paper is organized as follows.
We introduce our periodic resource allocation problem and the related dynamical system in Section~\ref{sec:model}.
More precisely, the optimization problem together with a simplified version based on an approximate functional are introduced in Subsection~\ref{sec:optgen}.
Some technical lemmas are given in Subsection~\ref{Sec:techres} and a criterion to guarantee that the original problem and its approximation share the same solution is given in Subsection~\ref{sec:criterion}. 
Some implementation remarks conclude this section in Subsection~\ref{Subsec:implementation_remarks}.
Section~\ref{sec:numexp} is devoted to the application to algal production. 
We present the models associated with the biological and the mixing device in a raceway pond in Subsection~\ref{sec:example}. 
The considered parameters are given in Subsection~\ref{subsec:param}. 
We illustrate the performance of our control strategies by numerical experiments in Subsection~\ref{subsec:test}. 
Finally, we conclude with some perspectives of our work in Section~\ref{sec:conclusion}.

\textbf{Notation. }
In what follows, $\N$ denotes the set of non-negative integers. 
The cardinal of a set $E$ is denoted by \#$E$.
Given a matrix $M$, we denote by $\text{ker}(M)$ its kernel, by $M^\top$ the corresponding transposed matrix and by $M_{i, j}$ its coefficient $(i , j)$.
In the same way, $W_n$ denotes the $n$-th coefficient of a vector $W$, whereas $\|W\|_\infty$ denotes its infinite norm, i.e. $\|W\|_\infty:=\max_n |W_n|$. 
The scalar product in $\R^N$, is denoted by $\langle \cdot,\cdot \rangle$, and we denote by $\I$ the identity matrix of size $N$. 

The set of permutations of $N$ elements, i.e, the set of bijections of $\{1,\cdots,N\}$, is denoted by $\S$.
The set of permutation matrices of size $N\times N$ is denoted by $\P$. 
Recall that a permutation matrix is a matrix which has exactly one entry equal to $1$ in each row and each column with the other entries being zero. 
A permutation matrix $P$ is associated to a permutation $\sigma$ by the formula $P_{i,j}=1$ if $i=\sigma(j)$ and $P_{i,j}=0$ otherwise. 
As consequence, if $W\in\R^N$, $(PW)_n=W_{\sigma^{-1}(n)}$ for all $n\in \{1,\cdots,N\}$.

\section{Description of the control problem and optimization}\label{sec:model}

Given a period $T$, an initial time $T_0$ and a sequence $(T_k)_{k\in\N}$, with $T_k:=k T+T_0$, we consider the following resource allocation process:
let $(r_n)_{n=1}^N\in\R^N$ representing a set of $N$ resources which are assumed to be constant over each time  interval $[T_k,T_{k+1})$ and renewed at each time $T_k$.
These resources can be allocated to $N$ activities denoted by $(x_n)_{n=1}^N$ where $x_n=x_n(t)$ consists of a real-valued function of time. 
Given a sequence of permutations $(\pi_k)_{k\in\N}$, with $\pi_k\in\S$, suppose that on the time interval $[T_k,T_{k+1})$, the resource $r_{\pi_k(n)}$ is assigned to the activity $x_{n}$, the latter evolving according to a linear dynamics
\begin{equation}\label{eq:evol}
\dot x_{n}(t)= -a(r_{\pi_k(n)}) x_{n}(t)+ b(r_{\pi_k(n)}),
\end{equation}
where $a : \R \rightarrow \R_+$ and $b : \R \rightarrow \R_+$ are given.

In this paper, we focus on an allocation strategy of the form $\pi_k=\sigma^k$, where $\sigma\in\S$ is fixed and $\sigma^k$ denotes the $k-$times repeated composition of $\sigma$ with itself. 
Such an assumption expresses that the same allocation device is used at each period of time. 
In this setting, the resource assignment process is such that at the end of each time period $[T_k,T_{k+1})$, the resource allocated to the activity $n$ is re-allocated to the activity $\sigma(n)$, or equivalently, that at the end of each time period $[T_k,T_{k+1})$, the resource $n$ is re-allocated to the activity $\sigma^{-1}(n)$.

Because the resource $(r_n)_{n=1}^N$ are constant with respect to time, the solution of~\eqref{eq:evol} can be computed explicitly. 
More precisely, denote by $x(t)\in\R^N$ the time dependent vector whose components are given by $x_n(t)$, i.e., $x(t):=(x_1(t),\cdots,x_N(t))^\top$. 
The process we consider reads
\begin{align}
x(t) =& \Delta(t)x(T_k) + \tau(t), \quad t\in[T_k,T_{k+1})\label{eq:sys1}\\
x(T_{k}) =& P x(T_{k}^-),\label{eq:sys2}
\end{align}
where $\Delta(t)$ is a time dependent diagonal matrix with $\Delta_{n n}(t) := e^{-a(r_{\sigma^k(n)})(t-T_k)}$, $\tau(t)$ is a time dependent vector with 
\begin{equation}\label{eq:V}
\tau_n(t):=\frac{b(r_{\sigma^k(n)})}{a(r_{\sigma^k(n)})}(1 - e^{-a(r_{\sigma^k(n)})(t-T_k)}),
\end{equation}
and $P\in\P$ the permutation matrix associated with $\sigma$. 
In this way, $k\in\N$ represents the number of re-assignments and $T_k^-$ represents the moment just before re-assignment.

\begin{remark}
All the results presented in this paper also hold for non-constant but $T-$periodic resources $(r_n(t))_{n=1}^N\in\R^N$.  
In the case of non-constant resources, the matrix $\Delta(t)$ and the vector $b(t)$ cannot be expressed explicitly. 
Such a technical issue can easily be handled using numerical integration and have no consequences for the ideas involved in analysis developed in this work. 
Hence, we consider constant resources for the clarity of the presentation.
\end{remark}

\subsection{Periodic control regime assumption}\label{subsec:periodic_regime}

Define $D:=\Delta(T)$ and $v:=\tau(T)$ and consider as a control the permutation matrix $P \in \P$ involved in~\eqref{eq:sys2}.
According to~(\ref{eq:sys1}--\ref{eq:sys2}), we have 
\begin{equation}\label{eq:xk}
x(T_{k+1}) =  P x(T_{k+1}^-) = P D x(T_k) + P v.
\end{equation}

In the next sections of this paper, we focus on a $T$-periodic solution of~(\ref{eq:sys1}--\ref{eq:sys2}). 
We will motivate this choice by two theorems. 
These require the following preliminary result.

\begin{lemma}\label{lem:inversible}
Given $k\in \N$ and $P\in\P$, the matrix $\I - (PD)^k$ is invertible.
\end{lemma}

\begin{proof}
Assume $\I - PD$ is not invertible, then there exists a non-null vector $X\in \text{ker}(\I - PD)$, which means $X = P D X$.
Let us denote $d_n = D_{n n}$, $n=1,\ldots, N$.
Denoting by $\sigma$ the permutation associated with $P$, we find that $(D X)_n = d_n X_n$ and $X_n = (P D X)_n = d_{\sigma(n)}X_{\sigma(n)}$.
In the same way, we have $X_n = \left((PD)^k X\right)_n = d_{\sigma^k(n)} \ldots d_{\sigma(n)}X_{\sigma^k(n)}$. 
Denoting by $K$ the order of $\sigma$, we have
\begin{equation*}
X_n = \big((PD)^K X\big)_n = d_{\sigma^K(n)} \ldots d_{\sigma(n)}X_{\sigma^K(n)} = d_{\sigma^K(n)} \ldots d_{\sigma(n)}X_n.
\end{equation*}
Since, $0<d_n<1$ for $n=1,\ldots, N$, then $0<d_{\sigma^K(n)} \ldots d_{\sigma(n)}<1$.
This implies that $X_n = 0$, which contradicts our assumption.
Therefore, $\I - PD$ is invertible. 
That $\I - (PD)^{k}$ is invertible for all $k>0$ can be proved in much the same way.
\end{proof}

We can now state a convergence result about $\left(x(T_k)\right)_{k\in\N}$.

\begin{theorem}\label{th:asympt}
There exists a unique $T-$periodic solution $x_{per}(t)$  of~(\ref{eq:sys1}--\ref{eq:sys2}), satisfying
\begin{equation}\label{eq:per_cond}
x_{per}(T_k) = (\I - PD)^{-1}Pv.
\end{equation}
Moreover, for any arbitrary initial condition $x(T_0)$, we have $\lim_{k\rightarrow+\infty} x(T_k) = x_{per}(T_k).$
\end{theorem}

\begin{proof}
The existence of a constant sequence $(x_{per}(T_k))_{k\in\N}$ satisfying~\eqref{eq:xk} follows from Lemma~\ref{lem:inversible}, applied with $k=1$. Solving~\eqref{eq:xk} in this setting gives~\eqref{eq:per_cond}. Let us then define the sequence $(e^k)_{k\in\N}$ by 
$e^k := x(T_k) - (\I - PD)^{-1}P v$. 
Since 
\begin{equation}\label{eq:rec}
e^{k+1} = (PD) e^k, 
\end{equation}
we find that
\[
    \|e^{k+1}\|_\infty=\|PDe^{k}\|_\infty   
                         =\|De^{k}\|_\infty  
                         \leq d_{\max}\|e^{k}\|_\infty, 
\]
where $d_{\max}:=\max_{n=1,\ldots,N}(d_{n})<1$. 
The result follows.
\end{proof}

This theorem shows that after a transient response, the system $x(t)$ can be correctly approximated by $x_{per}(t)$.
This steady state can be obtained in another way.

\begin{theorem}
We keep the notation of the previous lemma.
Given $k_0>0$, assume that the state $x$ is $k_0T$-periodic in the sense that after $k_0$ times of re-assignment, the state of each activity returns to its initial state $x_n(T_{k_0}) = x_n(T_0)$.
Then $x=x_{per}$.
\end{theorem}

\begin{proof}
Since $x$ is assumed to be $k_0T$-periodic, we have $e^0 = e^{k_0} = (PD)^{k_0} e^0 $. 
According to Lemma~\ref{lem:inversible}, $\I - (PD)^{k_0}$ is invertible, meaning that $e^0=0$. 
Combining this with~\eqref{eq:rec}, we get that $e^k = 0$, for $k\in\N$.
The result follows.
\end{proof}

A natural choice for $k_0$ would be the order $K$ of the permutation associated with $P$, which is the smallest integer greater than one such that $P^K=\I$.
Indeed, in this case $K$ is the minimal number of re-assignments required to recover the initial order of the components of $x$.
The previous result shows that every $KT-$periodic evolution will actually be $T-$periodic.
In the next section, we show that this property is decisive to formulate an optimization problem.
In addition, the computations to solve the optimization problem will be reduced, since the CPU time required to assess the quality of a permutation will not depend on its order.

\subsection{Objective function}\label{sec:optgen}

We still consider an arbitrary control  $P\in\P$ and the vector of activities $x(t)$ defined by~(\ref{eq:sys1}--\ref{eq:sys2}). 
Assume that the mean benefit of the process on the time period $[T_k,T_{k+1})$, i.e., after $k$ times of re-assignment, reads
\begin{equation}\label{eq:fk}
f^k := \langle w, \frac1T\int_{T_k}^{T_{k+1}} x(t) \D t\rangle,
\end{equation} 
where $w\in\R^N$ is a weighting vector expressing the relative importance of each activity.

Then the average benefit after $K$ re-assignment operations is given by
\begin{equation}\label{eq:abobjfun}
J_{av}:=\frac 1K\sum_{k=0}^{K-1} f^k.
\end{equation}
Such a formalization has been used by Cominetti \emph{et al.} in the context of forest maintenance and exploitation~\cite{Cominetti2009}. 
In this work, an infinite sum is considered to study the total benefit of all the re-assignment operations.
Replace now $x(t)$ in the benefit~\eqref{eq:fk}, by its expression~(\ref{eq:sys1}--\ref{eq:sys2}). We get
\begin{equation*}
f^k = \frac1T \langle w, \tilde D x(T_k) + \tilde v \rangle =  \frac1T   \left(\langle \tilde D w, x(T_k)\rangle + \langle w, \tilde v \rangle\right) ,
\end{equation*}
where $\tilde D_{n n} = \int_{T_k}^{T_{k+1}} \Delta_{nn}(t) \D t$ and $\tilde v_{n} = \int_{T_k}^{T_{k+1}} \tau_{n}(t) \D t$.
The only term  which depends on the re-assignment process is $x(T_k)$.

From now on, we focus of on the steady state introduced in Theorem~\ref{th:asympt}, meaning that we assume that $x=x_{per}$.
Because of~\eqref{eq:per_cond}, one finds 
\begin{equation*}
\langle \tilde D w, x_{per}(T_k)\rangle
= \langle \tilde D w, (\I - PD)^{-1}P v\rangle,
\end{equation*}
meaning that the benefit is the same for each re-assignment process.
As a consequence, $f^k$ does not depend on $k$ and that the average benefit $J_{av}$ (see~\eqref{eq:abobjfun}) satisfies
\[J_{av}(P)=\frac 1T \left(J(P) + \langle w, \tilde v \rangle\right),\]
where 
\begin{equation}\label{eq:Jpgen}
J(P) := \langle u, (\I - PD)^{-1}P v\rangle,
\end{equation}
with $u=\tilde D w$. 
It follows that maximizing $J_{av}$ with respect to $P$ is equivalent to maximizing $J$ with respect to $P$.  
Since $\#\P=\#\S= N!$, an exhaustive test of all the possible controls is out of range for large value of $N$. 
Hence, the maximization of $J$ cannot be tackled in realistic cases where a good numerical accuracy is required.
To overcome this difficulty, we propose in this section an approximation of this problem whose optimum can be determined explicitly, with a negligible computational cost.
For this purpose, we expand the functional~\eqref{eq:Jpgen} as follows
\begin{equation*}
\langle u, (\I - PD)^{-1}P v\rangle = \sum_{l=0}^{+\infty}\langle u, (PD)^{l}P v\rangle = \langle u, P v\rangle + \sum_{l=1}^{+\infty}\langle u, (PD)^{l}P v\rangle,
\end{equation*}
and consider as an approximation the first term of this series, namely
\begin{equation}\label{eq:Jpapproxgen}
J^{\text{approx}}(P):=\langle u, P v\rangle.
\end{equation}
Without loss of generality (see Appendix~\ref{app:usort} for the details), we assume that the entries of $u$ are sorted in ascending order, meaning that $u_1 \leq \ldots \leq u_N$.
Note that optimizing $J^{\text{approx}}$ amounts to solving an assignment problem~\cite{Burkard2008}.
Indeed, we have for example
\begin{equation*}
\min_{P\in\P} J^{\text{approx}}(P) = \min_{\sigma\in\S} \sum_{n=1}^N u_n v_{\sigma(n)}.
\end{equation*}
The latter expression reads as an assignment problem associated with the cost matrix\cite[p.5]{Burkard2008} $[u_iv_j]_{(i,j=1,\ldots,N)}$.
To make our exposition self-contained, we give the solution of this problem in Section~\ref{sec:criterion}.

\begin{remark}
A fairly common approach to deal with permutation matrices in discrete or combinatorial optimization is to relax the problem by extending the optimization to the set of bistochastic matrices. 
As an example, this technique corresponds to the Kantorovitch relaxation considered in optimal transport~\cite{Kantorovich1942} (see also~\cite{Bertsimas1997} for a more general presentation of the linear case, and~\cite{Lieb1981} for a similar strategy in the context of quantum chemistry). 
This approach allows the optimization to be performed by gradient-type methods. 
At the theoretical level, the goal is then to prove that the convergence takes place towards extremal points, i.e. permutation matrices. 
We have tested this approach on the nonlinear problem~\eqref{eq:Jpgen}. 
Our experiments indicate that the obtained limits are neither always permutation matrices nor optimal, which leads us to conjecture the existence of local (non-global) maxima for this extended form of $J$.

\end{remark}

\subsection{Some technical lemmas}\label{Sec:techres}

Let us state some preliminary properties about the permutation set $\S$ that we will use in the next section. 
Given $k\in\N$, and two arbitrary permutations $\s, \ts\in\S$, let us define 
\begin{equation*}
\begin{split}
E_k(\s,\ts):=&\left\{n=1,\ldots, N \ | \ \s^k(n)\neq\ts^k(n) \right\},\\
G_k(\s,\ts):=&\{ n=1,\ldots, N \ | \ \forall k'\leq k,\ \s^{k'}(n)=\ts^{k'}(n) \},
\end{split}
\end{equation*}
and 
$m_k :=\#E_k(\s,\ts)$.
We have the following result.

\begin{lemma}\label{lem:EkGk}
For $k\in \N$, we have $m_{k}\leq k m_1$ and $\# G_k(\s,\ts)\geq \max(N-k m_1,0)$. 
\end{lemma}

\begin{proof}
To shorten notation, we write in this proof $E_k$ instead of $E_k(\s,\ts)$, $E_{k+1}$ instead of $E_{k+1}(\s,\ts)$, $G_k$ instead of $G_k(\s,\ts)$, etc. 
From the definition of $E_k$, we have:
\begin{equation*}
E_{k+1}=\left((\{1,\ldots, N\}\setminus E_1)\cap E_{k+1}\right) \cup (E_1\cap  E_{k+1}).
\end{equation*}
The first subset in the right-hand side satisfies
\begin{equation*}
\begin{split}
\s\left((\{1,\ldots, N\}\setminus E_1)\cap E_{k+1}\right) =\ts\left((\{1,\ldots, N\}\setminus E_1)\cap E_{k+1}\right) \subset E_k,
\end{split}
\end{equation*}
so that $\#\left((\{1,\ldots, N\}\setminus E_1)\cap E_{k+1}\right) \leq \# E_k =: m_k$.

On the other hand, $ (E_1\cap E_{k+1}) \subset E_1 $, hence $\# (E_1\cap E_{k+1}) \leq m_1$. 
As a consequence, $m_{k+1}\leq m_k + m_1$. This implies $m_{k}\leq k m_1$.

As for $G_k$, we have:
\begin{equation}\label{eq:Ens}
G_{k}=(G_{k+1}\cap G_k)\cup (\s^{-k}(E_1)\cap G_k).
\end{equation}
Indeed, let $n\in G_k$, i.e, $\s^k(n)=\ts^k(n)$. If $\s^{k+1}(n)=\ts^{k+1}(n)$, then $n\in G_{k+1}$.
Otherwise, $\s^{k+1}(n)\neq\ts^{k+1}(n)$, meaning that $\s^{k+1}(n)\neq\ts(\s^{k}(n))$ which implies $\s^{k}(n)=\ts^{k}(n)\in E_1$, so that $n\in \s^{-k}(E_1)$.
This proves~\eqref{eq:Ens}, and we get as a by-product
\begin{equation*}
(G_{k+1}\cap G_k)\cap (\s^{-k}(E_1)\cap G_k)= \emptyset.
\end{equation*}
Moreover, since $G_{k+1}\subset G_k$,  we get $G_{k+1}\cap G_k= G_{k+1}$.
It follows that
\begin{equation*}
\# G_k=\# G_{k+1}+ \#\{\s^{-k}(E_1)\cap G_k\}.
\end{equation*}
Since $\#\{\s^{-k}(E_1)\cap G_k\}\leq \# E_1 = m_1$, we obtain $\# G_{k+1} \geq \# G_k - m_1$.
The result follows.
\end{proof}

In what follows, a transposition in $\S$ between two elements $i\neq j$ is denoted by $(i \ j)$.
By abuse of notation, $(n\ n)$ denotes the identity for all $n=1,\ldots, N$.
Given a permutation $\sigma\in \S$, we consider the sequence of permutations $(\sigma_n)_{n=0,\ldots, N}$ defined by
\begin{equation}\label{eq:sig}
\begin{split}
&\sigma_0=\sigma \\
&\sigma_n=(n\ \sigma_{n-1}(n))\circ \sigma_{n-1}.
\end{split}
\end{equation}
For all $n\leq N$, it immediately follows from this definition that
\begin{equation*}
\sigma_n|_{\{1,\ldots, n\}} = Id|_{\{1,\ldots, n\}} \text{ and } \sigma_{N-1} = \sigma_N = Id,
\end{equation*}
where $Id$ denote the identity permutation. 
Let us give two additional properties of this sequence.

\begin{lemma}\label{lem:prop1}
Let $\sigma\in \S$ and $(\sigma_n)_{n=1,\ldots, N-1}$ defined by~\eqref{eq:sig}. 
One has:
\begin{equation*}
\left\{i=1,\ldots, N \ | \ \sigma(i)=i\right\}=\left\{i=1,\ldots, N \ | \ \forall n=1,\ldots, N-1,\ \sigma_n(i)=i\right\}.
\end{equation*}
\end{lemma}

\begin{proof}
Given $i$ with $1\leq i \leq N$, such that $\s(i)=i$, let us prove that $\sigma_n(i)=i$ by induction on $n$. 
Since $\sigma_0=\sigma$, the result holds for $n=0$.
Suppose it holds at a rank $n-1$, meaning that $\s_{n-1}(i)=i$. 
By definition of $(\sigma_n)_{n=1,\ldots, N}$, one has:
\begin{equation*}
\sigma_n(i)= (n\ \sigma_{n-1}(n))\circ \sigma_{n-1}(i) =(n\ \sigma_{n-1}(n))(i).
\end{equation*}
If $i=n$, then $(n\ \sigma_{n-1}(n))(i)=\sigma_{n-1}(n)=\sigma_{n-1}(i)=i$.
If $i=\sigma_{n-1}(n)$, then $i=\sigma_{n-1}(i)=\sigma_{n-1}(n)$ and $i=n$, so that we conclude as in the previous case.
In the other cases, $\sigma_n(i)=\sigma_{n-1}(i)=i$. 
The result follows.
\end{proof}

\begin{lemma}\label{lem:prop2}
Let $i, j\in\{1,\ldots, N\}$, with $i < j$.
Let $\sigma\in\S$, with $\sigma = (i \ j)\circ\sigma'$, where $(i \ j)$ and $\sigma'\in\S$ have disjoint supports, i.e., $\sigma'(i)=i$ and $\sigma'(j)=j$. 
The sequence defined by~\eqref{eq:sig} satisfies $\sigma_{j}=\sigma_{j-1}$.
\end{lemma}

\begin{proof}
From~\eqref{eq:sig}, one has
\begin{equation*}
\sigma_{j} = (j \ \sigma_{j-1}(j)) \circ \sigma_{j-1}.
\end{equation*}
We need to prove that $\sigma_{j-1}(j)=j$.
Since $\sigma'$ and $(i \ j)$ are disjoint, then for $n<i$, $\sigma_n = (i \ j)\circ\sigma'_n$, where $\sigma'_n$ is  defined by~\eqref{eq:sig}, with the initial term $\s'_0=\s'$. 
In particular, $\sigma_n(i)=j$ for $n<i$.

In the case $n=i$, one has
\begin{equation*}
\sigma_{i} = (i \ \sigma_{i-1}(i)) \circ \sigma_{i-1} =  (i \ j) \circ \sigma_{i-1} = (i \ j)\circ(i \ j)\circ\sigma'_{i-1} = \sigma'_{i-1}.
\end{equation*}
In particular, $\sigma_i(j)=j$.

Finally, since $\s'_{i-1}(i)=i$, we find that $\s'_{i}=\s'_{i-1}$, and it follows by induction that for $n>i$, $\sigma_n = \sigma'_n$, which means $\sigma_n(j) = j$. 
In particular $\sigma_{j-1}(j) = j$.
This concludes the proof.
\end{proof}

The sequence $(\s_n)_{n=0,\ldots, N}$ can be used to decompose $J(\I)-J(P)$ for an arbitrary $P\in \P$, as stated in the next Lemma.

\begin{lemma}
Let $\s \in \S$ and $P\in\P$ the associated permutation matrix, we have:
\begin{equation*}
\langle u , (\I - P)v \rangle = \sum_{n=1}^{N-1} (u_n-u_{\sigma^{-1}_{n-1}(n)})(v_n-v_{\sigma_{n-1}(n)}).
\end{equation*}
\end{lemma}

\begin{proof}
Given $j\in\{0,\ldots,N\}$, define $S_j = \sum_{n=1}^{N}u_n v_{\sigma_j(n)}$.
Since $\sigma_j(n)$ and $\sigma_{j-1}(n)$ might only differ for $n=j$ and $n=\sigma^{-1}_{j-1}(j)$, we have
\begin{equation*}
\begin{split}
S_j-S_{j-1} &= \sum_{n=j}^N u_n(v_{\sigma_j(n)} - v_{\sigma_{j-1}(n)})\\
&= u_j(v_{\sigma_j(j)} - v_{\sigma_{j-1}(j)}) + u_{\sigma^{-1}_{j-1}(j)}(v_{\sigma_j(\sigma^{-1}_{j-1}(j))} - v_{\sigma_{j-1}(\sigma^{-1}_{j-1}(j))})\\
&=u_j(v_j - v_{\sigma_{j-1}(j)})+u_{\sigma^{-1}_{j-1}(j)}(v_{\sigma_{j-1}(j)} - v_j)\\
&=(u_j-u_{\sigma^{-1}_{j-1}(j)})(v_j - v_{\sigma_{j-1}(j)}).
\end{split}
\end{equation*}
The result then follows from $\langle u , (\I - P)v \rangle = S_{N-1}-S_0$.
\end{proof}

\subsection{Solutions of the optimization problems}\label{sec:criterion}

The previous lemma enables us to solve the problems $\max_{P\in\P}J^{\text{approx}}(P)$ and $\min_{P\in\P}J^{\text{approx}}(P)$.
Recall that the entries of $u$ are sorted in ascending order.

\begin{lemma}\label{lem:uPv}
Let $\sigma_{+},\ \sigma_{-}\in \S$ such that $v_{\sigma_{+}(1)} \leq v_{\sigma_{+}(2)} \cdots \leq v_{\sigma_{+}(N)}$ and $v_{\sigma_{-}(N)} \leq v_{\sigma_{-}(N-1)} \leq \cdots \leq v_{\sigma_{-}(1)}$ and $P_+,\ P_- \in \P$, the corresponding permutation matrices.
Then
\begin{equation*}
P_+ =  {\rm argmax}_{P\in \P} J^{\text{approx}}(P), \quad P_- =  {\rm argmin}_{P\in \P} J^{\text{approx}}(P).
\end{equation*}
\end{lemma}

\begin{proof}
Let $P\in \P$ and $\sigma\in \S$ the associated permutation, we have
\begin{equation}\label{eq:lowboundwrite}
\begin{split}
\langle u, (P_+ - P)v \rangle &= \langle u, (\I - PP_+^{-1})w
\rangle \\
                              &= \sum_{n=1}^{N-1} (u_n-u_{(\sigma'_{n-1})^{-1}(n)})(w_n-w_{\sigma'_{n-1}(n)}),
\end{split}
\end{equation}
where $w=(w_n)_{n=1}^N := (v_{\sigma_{+}(n)})_{n=1}^N$ and $\sigma'_{n}$ is the sequence defined by~\eqref{eq:sig} with $\sigma':=\sigma^{-1}_+ \circ \sigma$ the permutation associated with $PP_+^{-1}$.
Since $(w_n)_{n=1}^N$ by its definition is an increasing sequence, $\sigma'_{n-1}(n)\geq n$ and $(\sigma'_{n-1})^{-1}(n)\geq n$, we find that $\langle u, (P_+ - P)v \rangle\geq 0$.
The proof for the problem $\min_{P\in\P}\langle u ,P v \rangle$ is similar.
\end{proof}

We immediately deduce from this lemma that once $u$ and $v$ are given, the matrix $P_+$, $P_-$ of Lemma~\ref{lem:uPv} can be determined explicitly.
More precisely, $P_+$ is the matrix corresponding to the permutation which associates the largest coefficient of $u$ with the largest coefficient of $v$, the second-largest coefficient with the second-largest, and so on. 
In the same way, $P_-$ is the matrix corresponding to the permutation which associates the largest coefficient of $u$ with the smallest coefficient of $v$, the second-largest coefficient with the second-smallest, and so on.

\begin{remark}\label{rem:uvstrict}
The optimal matrices $P_+$ and $P_-$ are not unique as soon as either $u$ or $v$ contains at least two identical entries.
\end{remark}

We focus now on the case where $u$ as well as $v$ have entries with a constant sign.
Since the results in this section hold both for minimization and maximization problems, we can assume without loss of generality that $u, v$ are both positive.
Using the properties given in the previous section, we will show that in some cases, the problem $\max_{P\in \P}J(P)$ (resp. $\min_{P\in \P}J(P)$) and $\max_{P\in \P}J^{\text{approx}}(P)$ (resp. $\min_{P\in \P}J^{\text{approx}}(P)$) have the same solution.

We keep the notation of Lemma~\ref{lem:uPv}. 
Define for $n=1, \ldots, N$, 
\begin{equation}\label{eq:pn}
\widetilde{p}_n:=\min_{i , j=1,\ldots, N, i\neq n, j\neq n}|(u_n-u_i)(v_{\sigma_+(n)}-v_{\sigma_+(j)})|.
\end{equation}
Denote by $i_n$ and $j_n$ the solutions of the previous problem.
Since $u_n, v_{\sigma_+(n)}$ are sorted in ascending order, we find immediately that if $n=1$ (resp. $N$), then $i_n=j_n=2$ (resp. $i_n=j_n=N-1$).
Otherwise, $i_n=n-1$ or $i_n=n+1$, and the same result holds for $j_n$.
Sort $(\widetilde{p}_n)_{n=1}^N$ and denote by $(p_n)_{n=1}^N$ the resulting sequence, i.e., $p_1\leq,\ldots,\leq p_N$.
Define then for $m=1,\ldots,N$
\begin{equation}\label{eq:sm}
s_m := \sum_{n=1}^m p_n,
\end{equation}
and
\begin{equation}\label{eq:Fm}
F_m^-:=\sum_{n=1}^{\min(m,N)} u_n v_{\sigma_-(N-m+n)}, \quad F_m^+:=\sum_{n=\max(1,N-m+1)}^N u_n v_{\sigma_+(n)}.
\end{equation}
From the definition of these sequences, we have $F_m^+\geq F_m^-$. 
See Appendix~\ref{app:Fm} for the case where $u$ or $v$ negative. 
We are now in a position to give the main result of this section.

\begin{theorem}\label{thm:main}
Assume that $u$ and $v$ have positive entries and define
\begin{equation}\label{eq:phi}
\phi(m_1) := \frac 1{s_{\left \lceil{\frac{m_1}2}\right \rceil}} \Big(\sum_{l=1}^{+\infty} d_{\max}^l F_{(l+1)m_1}^+ - d_{\min}^l F_{(l+1)m_1}^-  \Big),
\end{equation}
where $m_1$ refers to the notation in Lemma~\ref{lem:EkGk}, $d_{\max}:=\max_{n=1,\ldots,N}(d_{n})$ and $d_{\min}:=\min_{n=1,\ldots,N}(d_{n})$.
Assume that:
\begin{equation}\label{eq:thmcon}
\max_{m_1\geq 2} \phi(m_1)\leq 1.
\end{equation}
Then the problem $\max_{P\in \P}\langle u ,(\I-P D)^{-1} P v \rangle$ (resp. $\min_{P\in \P}\langle u ,(\I-P D)^{-1} P v \rangle$) and the problem $\max_{P\in \P}\langle u ,P v \rangle$ (resp. $\min_{P\in \P}\langle u ,P v \rangle$) have the same solution.
\end{theorem}

\begin{proof}
We keep the notation in Section~\ref{Sec:techres} and give the proof in the case of the maximization problem.
The case of the minimization problem can be handled in the very same way.
Let $P\in \P$ and $\sigma\in \S$ the associated permutation, we have
\begin{align}
\langle u ,(\I- P_+D)^{-1} P_+ v \rangle - &\langle u ,(\I-P D)^{-1} P v \rangle \label{eq:line1} \\
                                    & = \sum_{l=0}^{+\infty} \langle u ,\big((P_+D)^l P_+-(PD)^l P\big) v \rangle  \nonumber\\
                                    &= \langle u ,(P_+-P)v \rangle + \sum_{l=1}^{+\infty} \langle u ,\big((P_+D)^l P_+-(PD)^l P\big) v \rangle. \label{eq:line3}
\end{align}
From the definition
$E_k(\s_+,\s)$ and $G_k(\s_+,\s)$, we have 
$E_1(\s_+,\s)\sqcup G_1(\s_+,\s)=\{1,\ldots,N\}$.
Let us denote by $(w_n)_{n=1}^N = (v_{\sigma_{+}(n)})_{n=1}^N$ and by $\sigma'_{n}$ the sequence defined by~\eqref{eq:sig} with $\sigma'_0:=\sigma^{-1}_+ \circ \sigma$.
From the definition of $E_1(\s_+,\s)$ and $G_1(\s_+,\s)$, we have $\sigma(G_1(\s_+,\s)) = \sigma_+(G_1(\s_+,\s))$ and $\sigma(E_1(\s_+,\s))=\sigma_+(E_1(\s_+,\s))$, which implies $\sigma_0'(E_1(\s_+,\s))=E_1(\s_+,\s)$, and for any $i\in G_1(\s_+,\s)$, $\sigma_0'(i) = i$. 
Using these properties and~\eqref{eq:lowboundwrite}, we have
\begin{equation}\label{eq:decom}
\begin{split}
\langle u ,(P_+-P)v \rangle = &\sum_{n=1}^{N-1} (u_n-u_{(\sigma'_{n-1})^{-1}(n)})(w_n-w_{\sigma'_{n-1}(n)})\\
                            = &\sum_{n\in E_1(\s_+,\s)} (u_n-u_{(\sigma'_{n-1})^{-1}(n)})(w_n-w_{\sigma'_{n-1}(n)}) \\
                            & + \sum_{n\in G_1(\s_+,\s)} (u_n-u_{(\sigma'_{n-1})^{-1}(n)})(w_n-w_{\sigma'_{n-1}(n)}) \\
                            = &\sum_{n\in E_1(\s_+,\s)} (u_n-u_{(\sigma'_{n-1})^{-1}(n)})(w_n-w_{\sigma'_{n-1}(n)}).
\end{split}
\end{equation}
In the case where there exists a transposition $(i \ i')$  with $i<i'$ in $\sigma'$, Lemma~\ref{lem:prop2} implies that  $u_{(\sigma'_{i'-1})^{-1}(i')}=u_{i'}$ and $w_{\sigma'_{i'-1}(i')}=w_{i'}$.
The maximum number of transpositions in $\sigma'_0$ is $\frac{m_1}2$ if $m_1$ is even, $\frac{m_1-3}2$ otherwise.
Hence, the smallest number of non-zero terms present in the last sum of~\eqref{eq:decom} is given by $m_1-\frac{m_1}2=\frac{m_1}2$ if  $m_1$ is even, $\frac{m_1-1}2$ otherwise.
In other words, there exists at least $\left \lceil{\frac{m_1}2}\right \rceil$ non-zero terms in the last sum of~\eqref{eq:decom}, which implies
\begin{equation}\label{eq:mino}
\langle u ,(P_+-P)v \rangle = \sum_{n\in E_1(\s_+,\s)} (u_n-u_{(\sigma'_{n-1})^{-1}(n)})(w_n-w_{\sigma'_{n-1}(n)}) \geq s_{\left \lceil{\frac{m_1}2}\right \rceil}.
\end{equation}
For $n\in\{1,\ldots, N\}$ and $l\in\N^*$, let us denote by $d_{\sigma,l,n}:=d_{\sigma^l(n)}d_{\sigma^{l-1}(n)}\cdots d_{\sigma(n)}$. 
Considering now the second term of the right-hand side of~\eqref{eq:line3}, we get
\begin{equation*}
<u,(P D)^l P v> = \sum_{n=1}^N u_n d_{\sigma^l(n)}d_{\sigma^{l-1}(n)}\cdots d_{\sigma(n)} v_{\sigma^{l+1}(n)} = \sum_{n=1}^N u_n d_{\sigma,l,n}v_{\sigma^{l+1}(n)}.
\end{equation*}
Using this notation and Lemma~\ref{lem:EkGk}, we find
\begin{equation}\label{eq:majo}
\begin{split}
|\langle u, (P_+ D)^l& P_+v - (P D)^l P v \rangle | \\
        &= \left|\sum_{n=1}^N u_n (d_{\sigma_+,l,n}v_{\sigma_+^{l+1}(n)} - d_{\sigma,l,n} v_{\sigma^{l+1}(n)}) \right|\\
        &= \left|\sum_{n\not\in G_{l+1}(\s_+,\s)} u_n (d_{\sigma_+,l,n}v_{\sigma_+^{l+1}(n)} -  d_{\sigma,l,n} v_{\sigma^{l+1}(n)}) \right|\\
        &= \left|\sum_{n\not\in G_{l+1}(\s_+,\s)} u_n d_{\sigma_+,l,n} v_{\sigma_+^{l+1}(n)} -  \sum_{n\not\in G_{l+1}(\s_+,\s)} u_n d_{\sigma,l,n} v_{\sigma^{l+1}(n)} \right|\\
        &\leq  d_{\max}^l \sum_{n\not\in G_{l+1}(\s_+,\s)} u_n v_{\sigma_+(n)} - d_{\min}^l \sum_{n\not\in G_{l+1}(\s_+,\s)} u_n  v_{\sigma_-(n)} \\
        &\leq  d_{\max}^l F^+_{(l+1)m_1} - d_{\min}^l F^-_{(l+1)m_1}.
\end{split}
\end{equation}
This result combined with~\eqref{eq:thmcon}, gives
\begin{equation*}
\begin{split}
\left|\sum_{l=1}^{+\infty} \langle u, (P_+ D)^l P_+v - (P D)^l P v  \rangle\right| \leq &\sum_{l=1}^{+\infty} d_{\max}^l F^+_{(l+1)m_1} - d_{\min}^l F^-_{(l+1)m_1} \\
\leq &s_{\left \lceil{\frac{m_1}2}\right \rceil}.
\end{split}
\end{equation*}
Considering now~\eqref{eq:mino}, we obtain
\begin{equation*}
\left|\langle u ,(P_+-P)v \rangle\right| \geq \left|\sum_{l=1}^{+\infty} \langle u, (P_+ D)^l P_+v - (P D)^l P v  \rangle\right|.
\end{equation*}
It follows that the first term of~\eqref{eq:line3} dominates the second one.
As a consequence, the former has the same sign as~\eqref{eq:line1}. 
The result follows.
\end{proof}

\subsection{Implementation remarks}\label{Subsec:implementation_remarks}

In this section, we give details on the practical computation of the infinite sum in~\eqref{eq:phi}. 
Given $m_1\in\{2,\ldots,N\}$, define by $l^*$ such that
\begin{equation*}
l^* := \left\lfloor\frac N{m_1}\right\rfloor -1.
\end{equation*}
We have
\begin{equation*}
\begin{split}
&\sum_{l=1}^{+\infty} \big(d_{\max}^l F^+_{(l+1)m_1} - d_{\min}^l F^-_{(l+1)m_1}\big)\\
&= \sum_{l=1}^{l^*} \big(d_{\max}^l F^+_{(l+1)m_1} - d_{\min}^l F^-_{(l+1)m_1}\big) + \sum_{l=l^*+1}^{+\infty} \big(d_{\max}^l F^+_{(l+1)m_1} - d_{\min}^l F^-_{(l+1)m_1}\big)\\
&= \sum_{l=1}^{l^*} \big(d_{\max}^l F^+_{(l+1)m_1} - d_{\min}^l F^-_{(l+1)m_1}\big) + \frac{d_{\max}^{l^*+1}}{1-d_{\max}} F_{N}^+ - \frac{d_{\min}^{l^*+1}}{1-d_{\min}} F_{N}^-.
\end{split}
\end{equation*}
It follows that the infinite sum involved in $\phi(m_1)$ actually reduces to a finite sum that can be computed numerically without any approximation.
As for the evaluation of $s_{\left \lceil{\frac{m_1}2}\right \rceil}$, only $\left \lceil{\frac{N}2}\right \rceil$ terms need to be computed.
Examples of behaviour of $s_m$ and $F^+_m,F^-_m$ are presented in Figure~\ref{fig:Fm}, whereas examples of behaviour of the function~\eqref{eq:phi} with respect to $m_1$ are shown in Figure~\ref{fig:criterion}.

\section{Application to algal production}\label{sec:numexp}

Algal raceway reactors are currently the most extended technology for microalgae growth, more than 90\% of the worldwide microalgae production is performed by this technology.
In this system, the algae are exposed to solar radiation and advected in a laminar regime. 
This regime holds as long as they remain far enough from the mixing device, that usually consists of a paddle wheel. 
Meanwhile, they evolve at a constant depth and one can consider that their vertical positions only change after passing through this device~\cite{Demory2018}.
Two main phenomena have to be taken into account to study algal production.
First, the photosynthetic activity of the algae close to surface may suffer from photoinhibition by which an excess of light decreases the speed of photosynthesis.
Second, the algae at the bottom of the raceway may not receive any light since this quantity exponentially decreases with respect to depth. 
In this framework, it has been shown that a well-chosen mixing device can increase significantly the algal growth by balancing the access to the light resource~\cite{Bernard2021a,Bernard2021c}.
In this section, we analyse this result in the framework of the resource allocation process and apply the theory developed in the previous section to an algal production case. 
Finally, we provide some numerical results to evaluate the efficiency of the mixing strategies and their approximation.

\subsection{Biological dynamics and raceway mixing modeling}\label{sec:example}

The growth of algae results from the activity of photosynthetic cells generated by solar radiation. 
This complex interaction is accurately described by the Han model~\cite{Han2002}, which takes into account the above-mentioned phenomenon of photoinhibition. 
In this model, each light harvesting unit is assumed to have three different states: open and ready to harvest a photon ($A$), closed while processing the absorbed photon energy ($B$), or inhibited if several photons have been absorbed simultaneously leading to an excess of energy ($C$).
Their dynamics is described by the following system
\begin{equation*}
\left\{
\begin{array}{lr}
\dot{A} = -\sigma_H I A + \dfrac 1{\tau_H}B,\\
\dot{B} =  \sigma_H I A - \dfrac 1{\tau_H}B + k_r C - k_d\sigma_H I B,\\
\dot{C} = -k_r C + k_d \sigma_H I B.
\end{array}
\right.
\end{equation*}
Here $A, B$ and $C$ are the relative frequencies of the three possible states with $A+B+C=1$, and $I$ is a continuous time-varying signal representing the photon flux density. 
The coefficient $\sigma_H$ stands for the specific photon absorption, $\tau_H$ is the turnover rate, $k_r$ represents the photosystem repair rate and $k_d$ is the damage rate.

Since the sum of these three states is equal to one, the latter system can be reduced to two equations by eliminating $B$ as following:
\begin{equation*}
\begin{split}
\begin{pmatrix}
\dot A\\
\dot C
\end{pmatrix}
=& \epsilon \cdot (M_H
\begin{pmatrix}
A\\
C
\end{pmatrix}
+ b_H), \quad
\epsilon=\begin{pmatrix}
1 & 0\\
0 & k_d
\end{pmatrix}, \\
&M_H=\begin{pmatrix}
-(\sigma_H I + \frac1{\tau_H}) & -\frac1{\tau_H}\\
-\sigma_H I & -(\frac{k_r}{k_d} +\sigma_H I)
\end{pmatrix}, \quad 
b_H=\begin{pmatrix}
\frac1{\tau_H} \\
\sigma_H I
\end{pmatrix}.
\end{split}
\end{equation*}
The dynamics of the open state $A$ reaches its steady state following a process whose speed is much higher than the dynamics of the photoinhibition state $C$. 
This phenomenon is mainly due to the presence of the multiplicative parameter $k_d$ which is on the order of $10^{-4}$ whereas the absolute value of the entries of $M_H$ and $b_H$ are on the order of $0.1-6$ (see Table~\ref{Tab1}, where an example of typical values for the Han parameters is given).
We can then apply a slow-fast approximation using singular perturbation theory~\cite{Lamare2019}. 
This approach consists in replacing $A$ by its pseudo steady state $A_{\text{steady}}:=\frac{1-C}{\tau_H\sigma_H I+1}$ into the reduced evolution equation of $C$. 
The Han dynamics can then be reduced to a single evolution equation on $C$:
\begin{equation}\label{eq:dotC}
\dot{C} = -\alpha(I) C + \beta(I),
\end{equation}
where
\begin{equation*}
\alpha(I) := k_d\tau_H \frac{(\sigma_H I)^2}{\tau_H \sigma_H I+1} + k_r, \quad \beta(I) := k_d\tau_H \frac{(\sigma_H I)^2}{\tau_H \sigma_H I+1}.
\end{equation*}

The specific growth rate is proportional to $\sigma_H I A$, replacing $A$ by its pseudo steady state $A_{\text{steady}}$, the net specific growth rate is defined by
\begin{equation}\label{eq:mu}
\mu(C,I) := -\gamma(I)C + \zeta(I),
\end{equation}
where
\begin{equation*}
\gamma(I)  := \frac{k_H\sigma_H I}{\tau_H \sigma_H I+1}, \quad \zeta(I) := \frac{k_H\sigma_H I}{\tau_H \sigma_H I+1}-R.
\end{equation*}
Here, $R$ denotes the respiration rate and $k_H$ is a factor which relates the photosynthetic activity to the growth rate.

We assume that the system is perfectly mixed so that the biomass concentration is homogeneous.
Meanwhile, we also assume that the photosynthetic units grow slowly such that the variations of biomass concentration and background turbidity are negligible over one lap of the raceway. 
At this timescale, the turbidity and biomass concentration can be supposed to be constant. 
In this framework,  the light intensity reads as a function of the depth $z$ and can be modeled by the Beer-Lambert, i.e.,  
\begin{equation}\label{eq:Beer}
I(z) = I_s\exp(\varepsilon z),
\end{equation}
where $I_s$ is the light intensity at the free surface and $\varepsilon$ is the light extinction coefficient.
The average net specific growth rate over the domain is then defined by
\begin{equation}\label{eq:mubar}
\bar{\mu} := \frac 1T\int_0^T\frac 1{h}\int_{-h}^{0} \mu\big(C(t, z), I(z)\big) \D z \D t,
\end{equation}
where $h$ is the depth of the raceway pond and $T$ is the average duration of one lap of the raceway pond.

Let us now see how this model can be included in the framework of Section~\ref{sec:model}. 
In order to compute numerically~\eqref{eq:mubar}, we introduce a vertical discretization of the fluid, consisting of $N$ layers uniformly distributed on a vertical grid. 
The depth of the layer $n$ is given by
\begin{equation}\label{eq:zn}
z_n := -\frac {n-\frac12}{N} h, \quad n=1,\ldots, N.
\end{equation}
For a given initial photoinhibition state $C_n(0)$ associated with a photosystem located in layer $n$, let $C_n(t)$ be the solution of~\eqref{eq:dotC} at time $t$.
In this semi-discrete setting, the average net specific growth rate in the raceway pond can be defined by
\begin{equation}\label{eq:muN}
\bar \mu_{N} := \frac 1T \int_0^T \frac 1{N}\sum_{n=1}^{N}\mu( C_n(t),I_n ) \D t,
\end{equation}
where $I_n$ is the light intensity received in the layer $n$. 
The solution of~\eqref{eq:dotC} can be computed explicitly to get a formula that takes the form of~\eqref{eq:sys1}. 
Denoting by $C(t)$ the time dependent vector whose components are given by $C_n (t)$, it follows that $\bar \mu_{N}$ satisfies
\begin{equation}\label{eq:muNexp}
\bar \mu_N=\frac 1{NT} \Big(\langle \Gamma, C(0) \rangle + \langle \mathbf 1,Z\rangle\Big),
\end{equation}
where $\mathbf 1$ is a vector of size $N$ whose coefficients are equal to 1, and $\Gamma, Z$ are two vectors such that
\begin{equation*}
\begin{split}
\Gamma_n&:=\frac{\gamma(I_n)}{\alpha(I_n)}(e^{-\alpha(I_n)T}-1),\\
Z_n&:= \frac{\gamma(I_n)\beta(I_n)}{\alpha(I_n)^2}(1-e^{-\alpha(I_n)T}) - \frac{\gamma(I_n)\beta(I_n)}{\alpha(I_n)}T+ \zeta(I_n) T.
\end{split}
\end{equation*}
The detail of the computations giving rise to~\eqref{eq:muNexp} is presented in Appendix~\ref{app:exacomp}.

Assume now that at each lap, the layers are permuted according to $\sigma\in\S$, meaning that the algae at the layer $n_1$ are supposed to be entirely transferred to the layer $n_2=\sigma(n_1)$ when passing through the mixing device, 
This mixing process is depicted schematically on an example in Figure~\ref{fig:P}.
\usetikzlibrary{decorations.pathreplacing}
\begin{figure}[hptb]
\begin{center}
\begin{tikzpicture}
\node at (0,-0.7)[anchor=north] {$T_k$};
\node at (2,-0.7)[anchor=north] {$T_{k+1}^{-}$};
\node at (3,-0.7)[anchor=north] {$T_{k+1}$};
\node at (5,-0.7)[anchor=north] {$T_{k+2}^{-}$};
\draw [->,thick] (0,-0.5) -- (2,-0.5);
\draw [->,thick] (3,-0.5) -- (5,-0.5);
\draw [dashed] (0,0.5) -- (2,0.5);
\draw [dashed] (3,0.5) -- (5,0.5);
\draw [dashed] (0,1.5) -- (2,1.5);
\draw [dashed] (3,1.5) -- (5,1.5);
\draw [dashed] (0,2.5) -- (2,2.5);
\draw [dashed] (3,2.5) -- (5,2.5);
\draw [dashed] (0,3.5) -- (2,3.5);
\draw [dashed] (3,3.5) -- (5,3.5);
\draw [decorate,decoration={brace,amplitude=4pt},thick](-0.25,-0.45) -- (-0.25,0.45);
\draw [decorate,decoration={brace,amplitude=4pt},thick](-0.25,0.55) -- (-0.25,1.45);
\draw [decorate,decoration={brace,amplitude=4pt},thick](-0.25,1.55) -- (-0.25,2.45);
\draw [decorate,decoration={brace,amplitude=4pt},thick](-0.25,2.55) -- (-0.25,3.45);
\node at (-0.7,0)[anchor=east]{Layer four };
\node at (-0.7,1)[anchor=east]{Layer three };
\node at (-0.7,2)[anchor=east]{Layer two };
\node at (-0.7,3)[anchor=east]{Layer one };
\draw [thick](0,0) -- (2,0);
\draw [thick](3,0) -- (5,0);
\draw [thick](0,1) -- (2,1);
\draw [thick](3,1) -- (5,1);
\draw [thick](0,2) -- (2,2);
\draw [thick](3,2) -- (5,2);
\draw [thick](0,3) -- (2,3);
\draw [thick](3,3) -- (5,3);
\draw [dotted,thick,->](2,0) -- (2.9,3);
\draw [dotted,thick,->](2,1) -- (2.9,0);
\draw [dotted,thick,->](2,2) -- (2.9,1);
\draw [dotted,thick,->](2,3) -- (2.9,2);
\node at (2.5,4.5) (P){$\sigma$};
\draw [->,thick](P)-- (2.5,3.75);
\draw [->,thick](5.5,-0.6) -- (5.5,4.5);
\node at (5.9,4.5) {$z$};
\node at (5.9,3.5) {0};
\node at (5.9,-0.5){$-h$};
\node at (6.5,3) {$z_1=z_{\sigma(4)}$};
\node at (6.5,2) {$z_2=z_{\sigma(1)}$};
\node at (6.5,1) {$z_3=z_{\sigma(2)}$};
\node at (6.5,0) {$z_4=z_{\sigma(3)}$};
\end{tikzpicture}
\end{center}
\caption{Schematic representation of the mixing process over two laps. 
Here, the vertical discretization number $N=4$ and the mixing device corresponds to the cyclic permutation $\sigma= (1 \ 2 \ 3\ 4)$.}
\label{fig:P}
\end{figure}
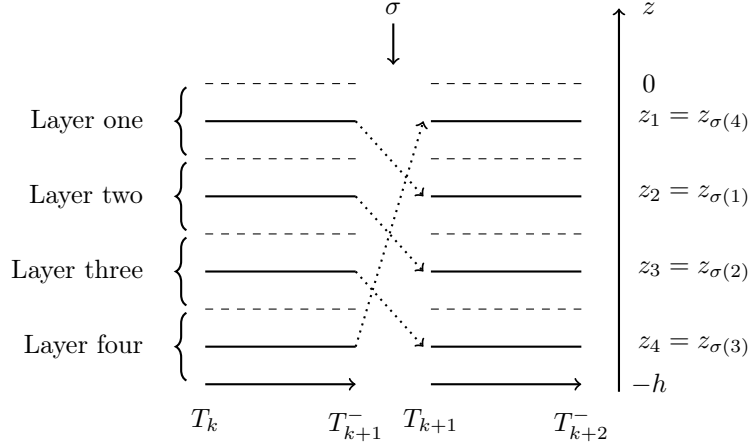
The interest of such a device is to mix the algae to better balance their exposure to light and increase the production.
In this way, though the light resource is considered to be constant, this mixing model and the constraints on the light level received at each layer make the process equivalent to a resource allocation problem.

At this step, we see that up to a change of notation, this model can be interpreted in terms of the resource allocation process described in Section~\ref{sec:model}. 
More precisely, the activity $x_n(t)$ corresponds to the photoinhibition state $C_n(t)$, the resource $r_n$ correspond to the light intensity $I_n$ and the functions $a,b$ correspond to $\alpha,\beta$.

Considering now the periodic regime, we get thanks to Theorem~\ref{th:asympt} and~\eqref{eq:per_cond} that the steady state of this system satisfies $C(0) =  (\I - PD)^{-1}P V$, where $D$ and $V$ are a diagonal matrix and a vector of size $N\times N$ and $N$ respectively, whose components are given by
\begin{equation}\label{def:D}
D_{nn}:= e^{-\alpha(I_n)T},\quad V_n := \frac{\beta(I_n)}{\alpha(I_n)}(1-e^{-\alpha(I_n)T}).
\end{equation}
Our goal is to find the control, i.e. a permutation $\sigma$, which maximizes the average growth rate $\bar \mu_N$. 
Since only $C(0)$ in~\eqref{eq:muNexp} depends on the control $P$, we find that the objective function $\bar \mu_N=\langle \Gamma, (\I - PD)^{-1}P V \rangle$ has the form of $J(P)$ in~\eqref{eq:Jpgen} with $u = \Gamma$, $v = V$ and $D$ defined in~\eqref{def:D}.

\subsection{Parameter settings}\label{subsec:param}

Consider a raceway whose water elevation $h=\SI{0.4}{m}$, which corresponds to typical raceway pond setting.
All the numerical parameters values considered in this section for Han's model are taken from~\cite{Grenier2020} and recalled in Table~\ref{Tab1}.
\begin{table}[htbp]
\caption{Parameter values for Han Model }
\label{Tab1}
\begin{center}
\begin{tabular}{|c|c|c|}
\hline
$k_r$ & 6.8 $10^{-3}$ & $\si{s^{-1}}$\\
\hline
$k_d$ & 2.99 $10^{-4}$  & -\\
\hline
$\tau$ & 0.25 & $\si{s}$\\
\hline
$\sigma_H$ & 0.047 & $\si{m^2.\mu mol^{-1}}$\\
\hline
$k_H$ & 8.7 $10^{-6}$ & -\\
\hline
$R$ & 1.389 $10^{-7}$ & $\si{s^{-1}}$\\
\hline
\end{tabular}
\end{center}
\end{table}
Recall that $I_s$ is the light intensity at the free surface.
In order to fix the value of the light extinction coefficient $\varepsilon$ in~\eqref{eq:Beer}, we assume that only a fraction $q$  of $I_s$ reaches the bottom of the raceway pond, meaning that $I_b = q I_s$, where $q\in[0,1]$ and $I_b$ is the light intensity at the bottom.
It follows that $\varepsilon$ can be computed by
\begin{equation*}
\varepsilon = (1/h)\ln(1/q).
\end{equation*}
In practice, this quantity can be implemented in the experiments by adapting the biomass harvesting frequency, or the dilution rate for continuous cultivation. 
In what follows, the varying parameters are $I_s$, the ratio $q$ and $T$.
We consider $I_s\in[0,2500] \si{ \mu mol.m^{-2}.s^{-1}}$, $q\in[0.1\%,10\%]$ and $T\in[1,1000] \si{s}$.
The number of layers $N$ remains small as we need to test numerically $N!$ permutation matrices for each triplet $(I_s, q, T)$.

\subsection{Numerical tests}\label{subsec:test}

As shown in~\cite[Section IV.B]{Bernard2021a}, Problem~\eqref{eq:Jpgen} admits non-trivial optimal permutation strategies which may significantly change according to the parameter settings. 
In this section, we study and compare the true and the approximated solutions as well as their efficiency with respect to the average net specific~\eqref{eq:mubar}.

We start by investigating some properties of the items defined in the previous sections. 
Recall that the two sequences $u, v$ used in Section~\ref{sec:optgen} correspond in our application to $\Gamma, V$ respectively.
We consider $N=20$ layers and two parameters triplets, namely $(I_s, q, T) = (2000,5\%,1000)$ and (800,0.5\%,1).
Figure~\ref{fig:gammaV} shows the evolution of these two quantities as a function of $I$.
\begin{figure}[htpb]
\centering
\includegraphics[scale=0.28]{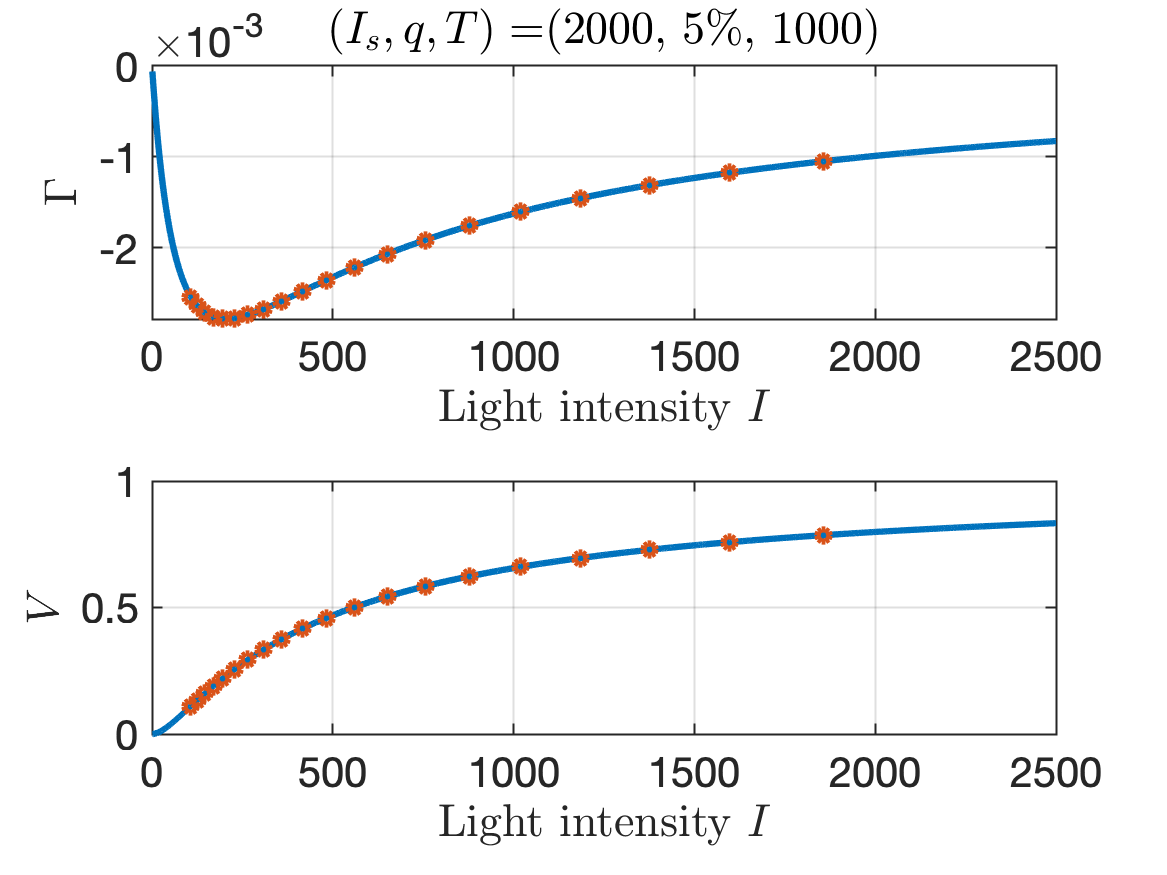}
\includegraphics[scale=0.28]{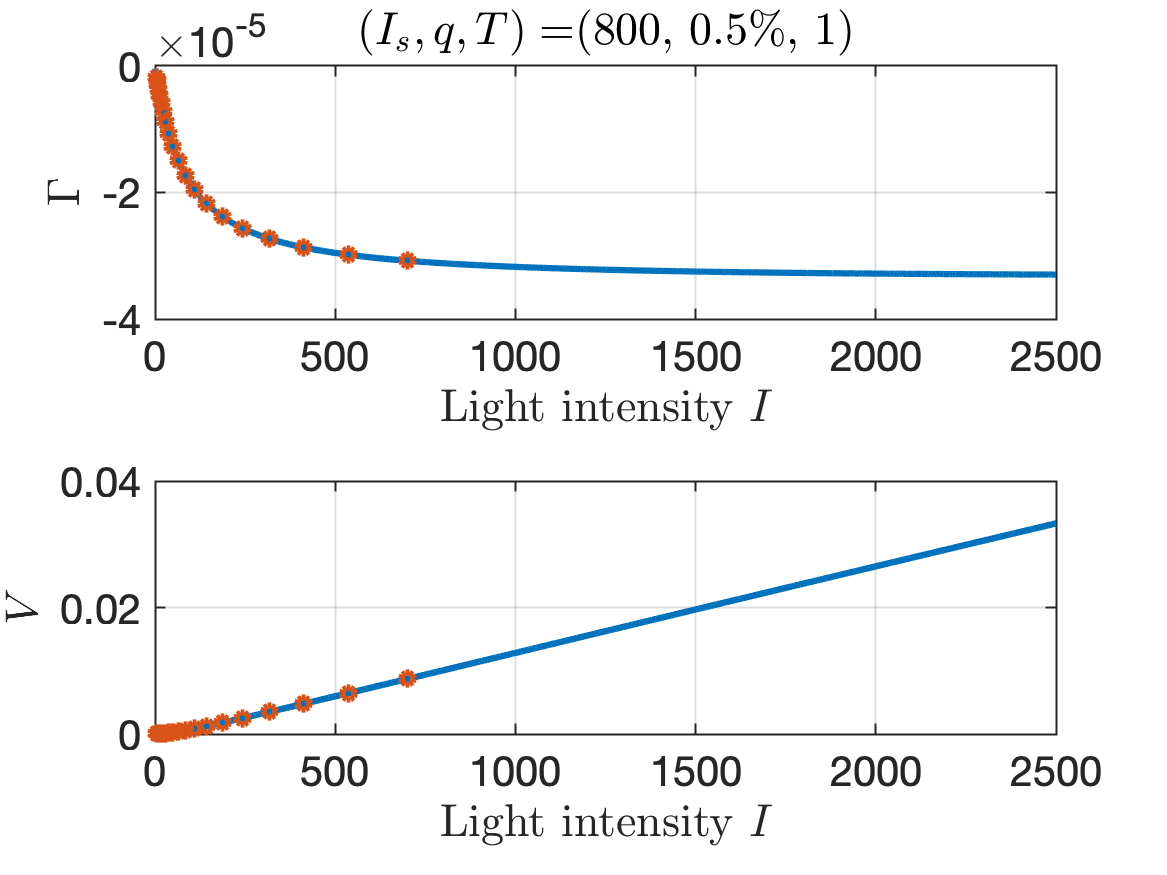}
\caption{$\Gamma$ and $V$ with respect to the light intensity $I$ (Blue curve). 
Discretisation points (Red point) chosen for $(I_s, q, T) = (2000,5\%,1000)$ (Left) and (800,0.5\%,1) (Right).}
\label{fig:gammaV}
\end{figure}
Note that in both cases, $V$ is positive with sorted entries, as it can be seen in~\eqref{eq:V}. 
On the contrary, the discretized $\Gamma$ is negative and not necessarily sorted.
We refer to Appendix~\ref{app:exacomp} for more details about $V$ and $\Gamma$.

We then study the behaviour of the sequences $F^+_m, F^-_m, s_m$ and $\phi(m_1)$ defined in Section~\ref{sec:criterion} for the same two parameters triplets.
Note that since $\Gamma$ is negative, $F^-_m$ and $F^+_m$ are defined as in Appendix~\ref{app:Fm} (and not as in~\eqref{eq:Fm}).
We choose $N=7$ and $N=20$ to check the performance for two different discretisation numbers of layers.
\begin{figure}[htpb]
\centering
\includegraphics[scale=0.3]{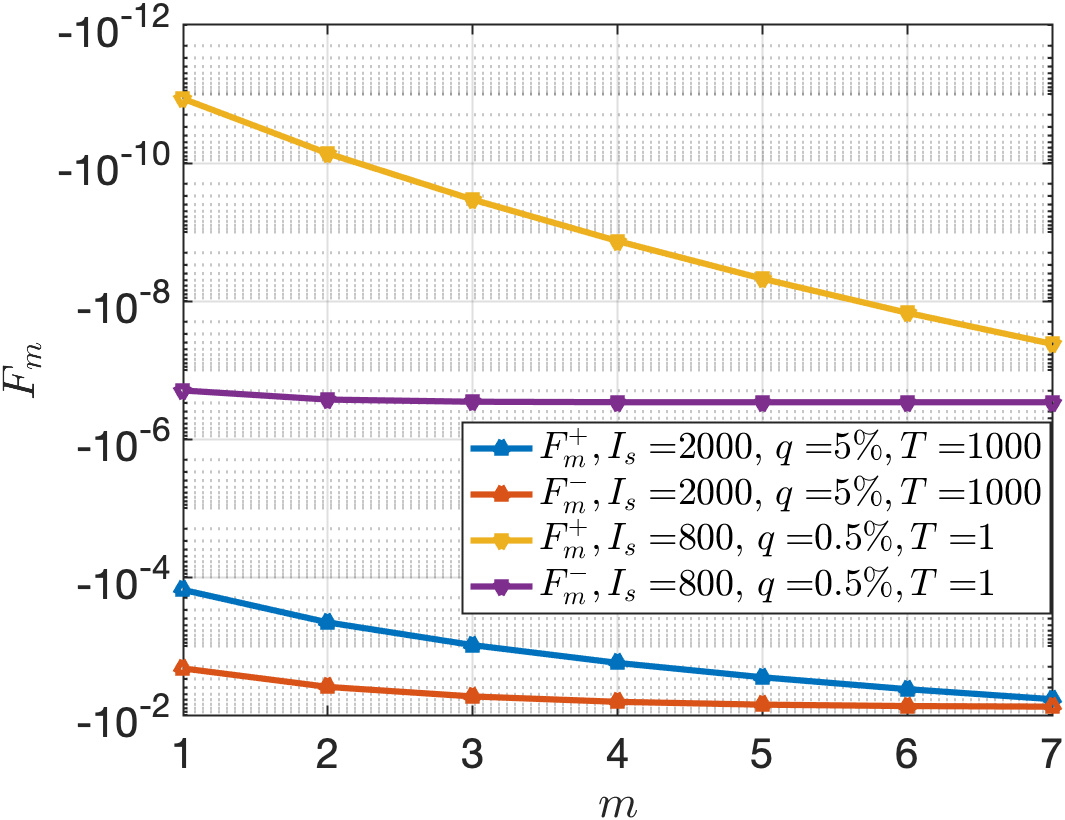}
\includegraphics[scale=0.3]{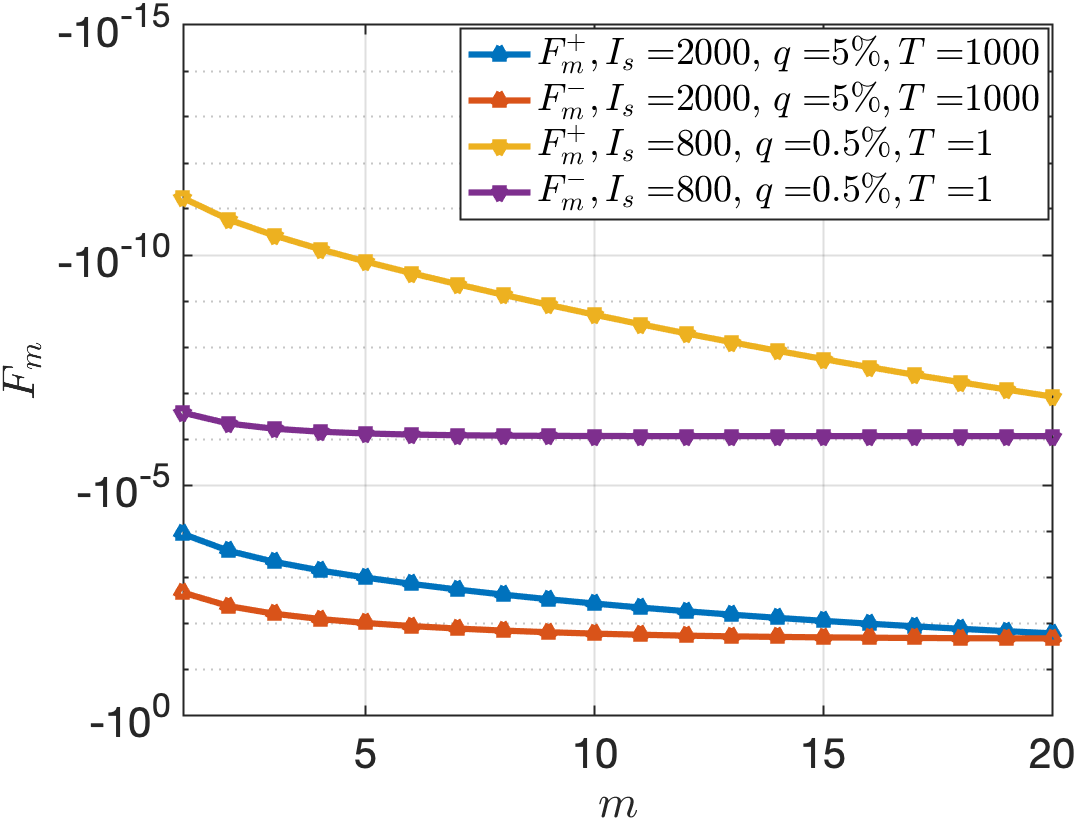}\\
\includegraphics[scale=0.3]{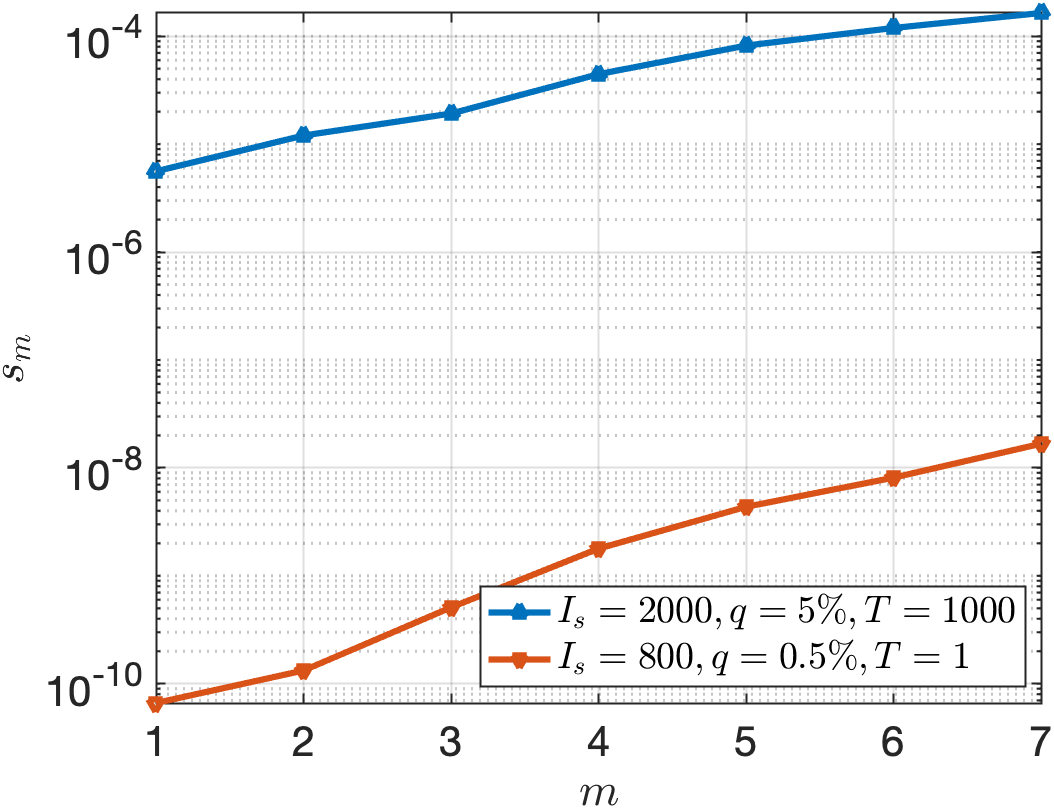}
\includegraphics[scale=0.3]{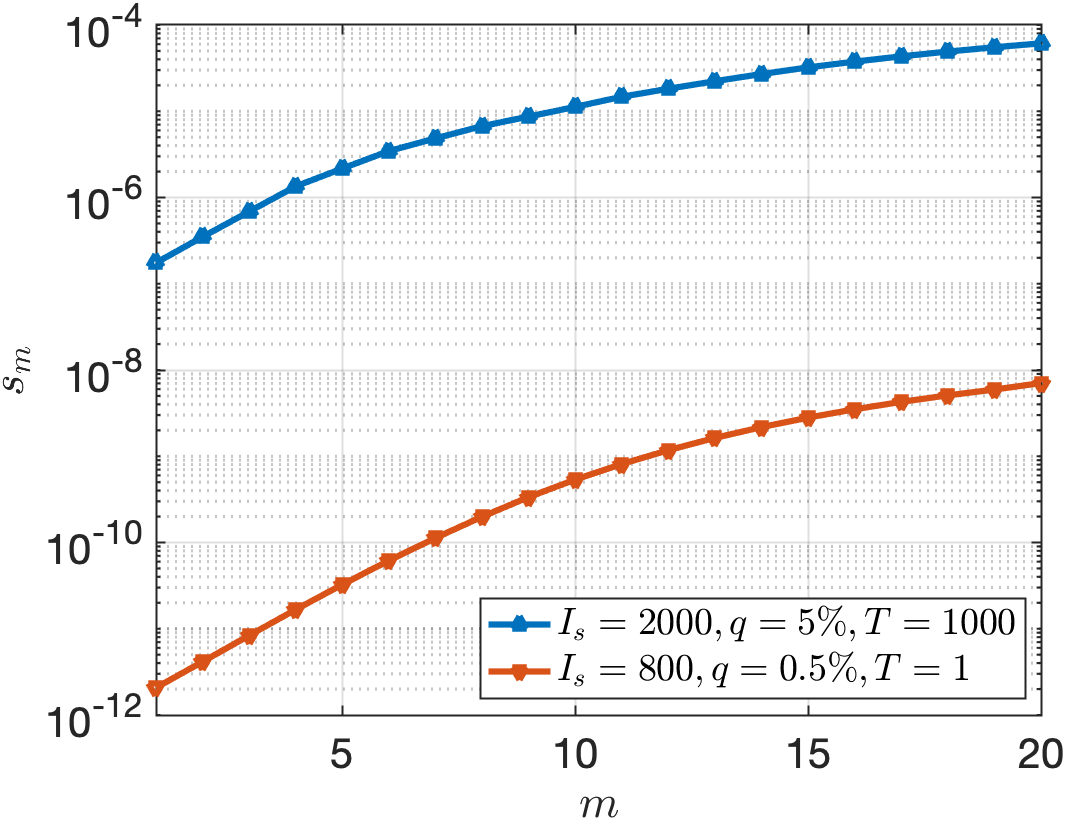}
\caption{Example of sequences $F_m^+$, $F_m^-$ (Top) and $s_m$ (Bottom) with respect to $m$ for the two parameters triplets.
Left: $N=7$. 
Right: $N=20$.
}
\label{fig:Fm}
\end{figure}
One can see in Figure~\ref{fig:criterion} that the maximal value of $\phi(m_1)$ is always obtained for $m_1=2$, and that the maximal value $\phi(m_1)$ appears to be an increasing function of $N$.
This makes the criterion given in Section~\ref{sec:criterion} less efficient for many layers $N$.
Further analysis is required to obtain a criterion that does not depend on $N$.
\begin{figure}[htpb]
\centering
\includegraphics[scale=0.3]{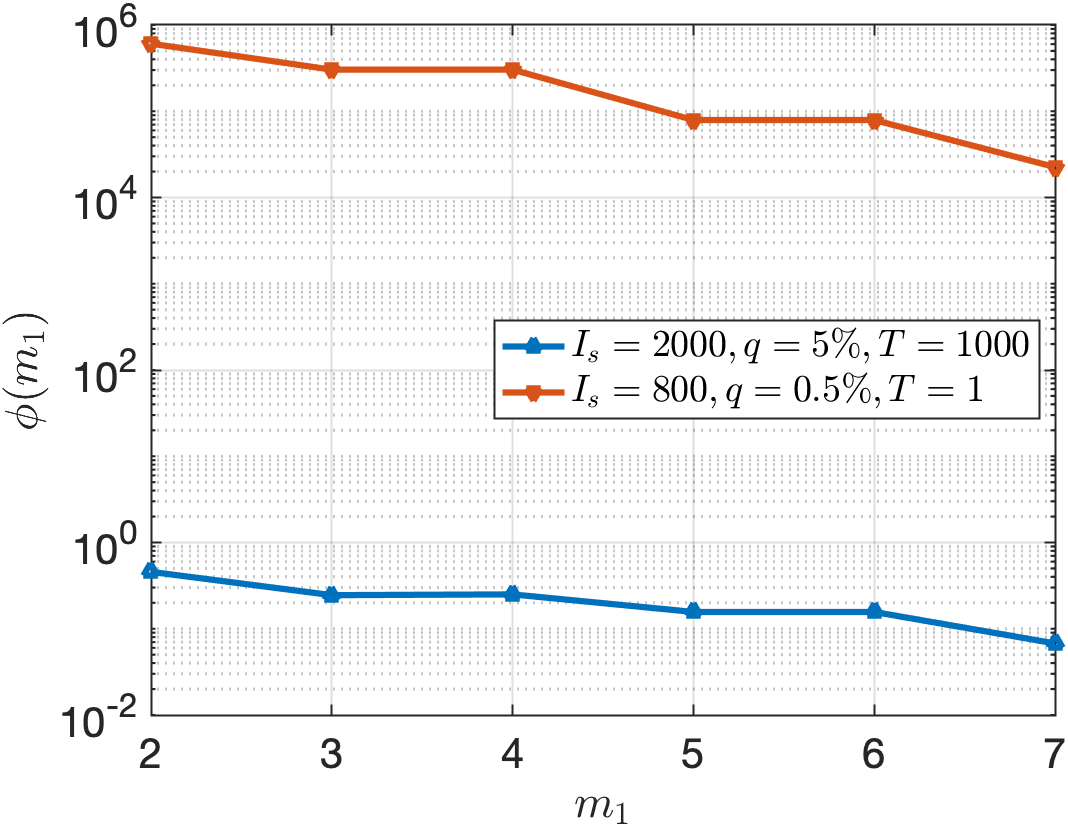}
\includegraphics[scale=0.3]{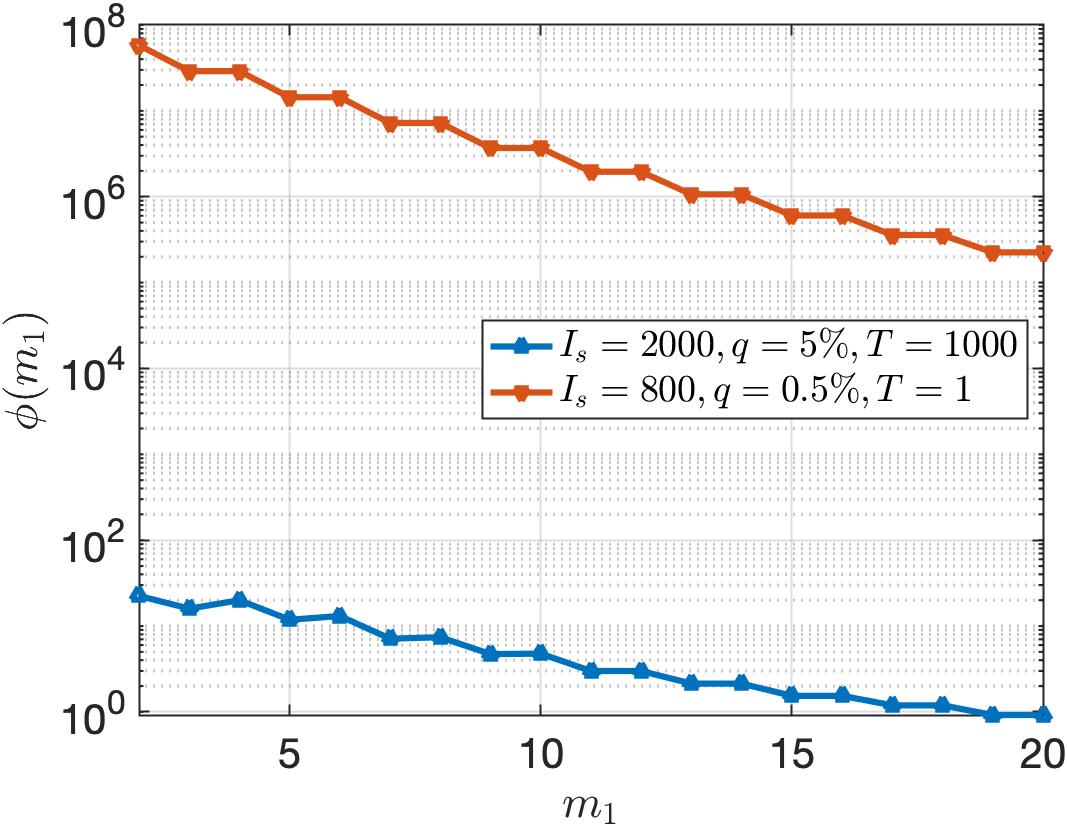}
\caption{Example of behaviour of $\phi(m_1)$ with respect to $m_1$ for two parameters triplets and two different $N$.
Left: $N=7$.
Right: $N=20$.}
\label{fig:criterion}
\end{figure}

The next test is devoted to the convergence of the average growth rate $\bar \mu_N$ with respect to the  number of layers $N$.
We keep the two triplets of parameters of the previous test.
Due to the limit of the computer memory, the computation of $\bar \mu_N(P_{\max})$ is tractable for small values of $N$, in our case lower than or equal to $N=11$.
Such an issue does not occur in the case of $\bar \mu_N(P_+)$.
Figure~\ref{fig:muN} presents the behaviour of $\bar \mu_N$.
\begin{figure}[htpb]
\centering
\includegraphics[scale=0.4]{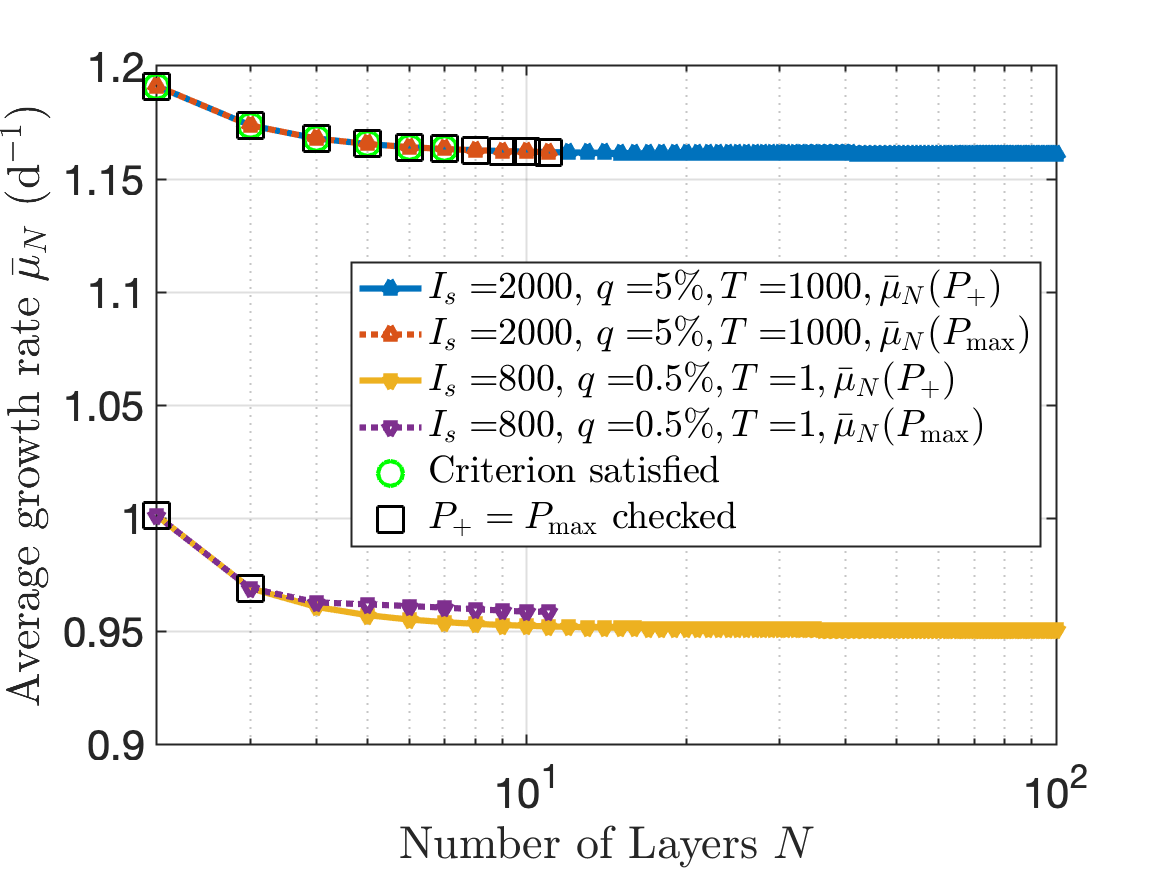}
\caption{Average growth rate $\bar \mu_N$ obtained with $P_{\max}$ and $P_+$ as a function of $N$ for the two parameters triplets.
The green circles mark the case when the criterion is satisfied.
The black squares mark the case when $P_{\max} = P_+$ is observed.
}
\label{fig:muN}
\end{figure}
For the parameter triplet (2000, 5\%,1000), the criterion is satisfied until $N=8$ (green circle), which is confirmed in Figure~\ref{fig:criterion} (Left) for $N=7$ where the maximal value of $\phi(m_1)$ is already close to 1.
Though the criterion is not satisfied for $N>8$, we observe that $P_+ = P_{\max}$ from $N=2$ to $N=11$. 
As for the triplet (800,0.5\%,1), one can see that $P_+ = P_{\max}$ until $N=3$.
Figure~\ref{fig:Popt} shows the optimal control strategies for these two different parameter triplets in the case $N=11$ and $N=100$.
\begin{figure}[htpb]
\centering
\includegraphics[scale=0.35]{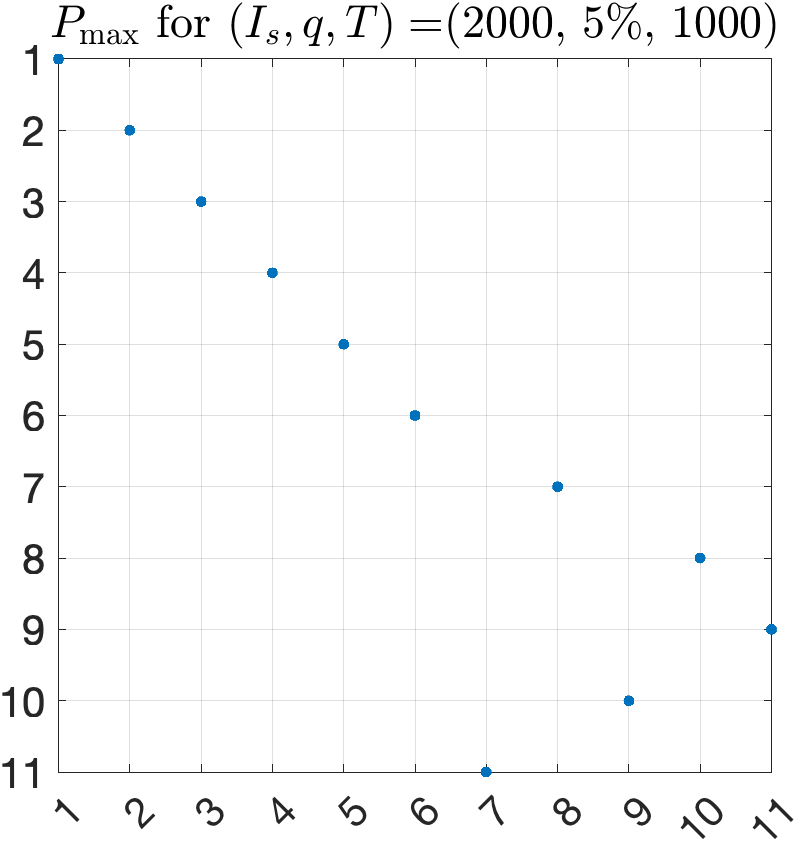}
\includegraphics[scale=0.35]{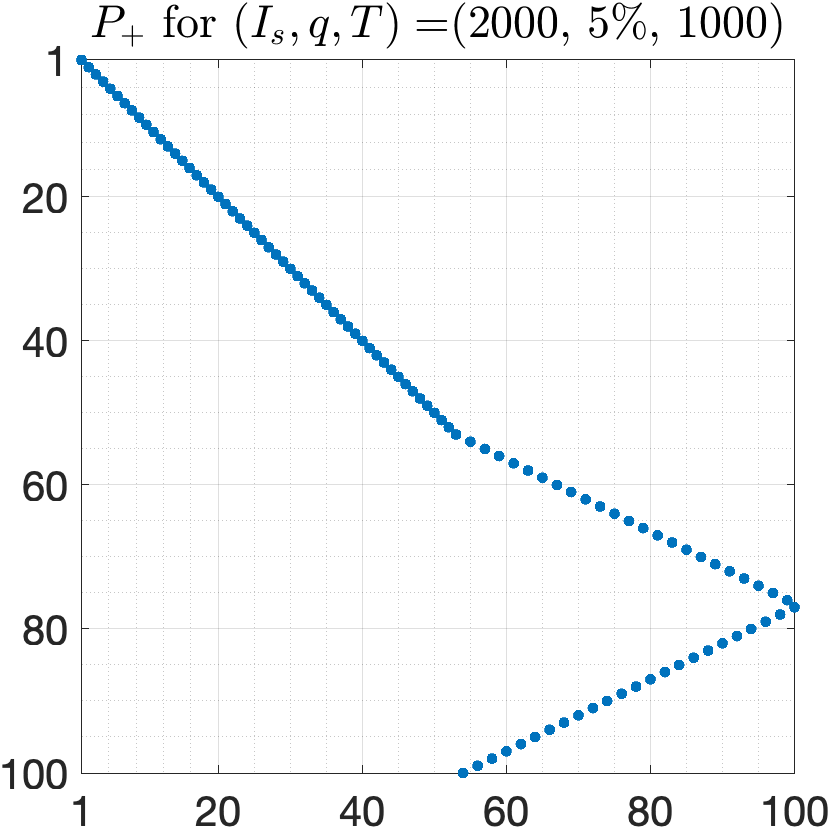}\\
\includegraphics[scale=0.35]{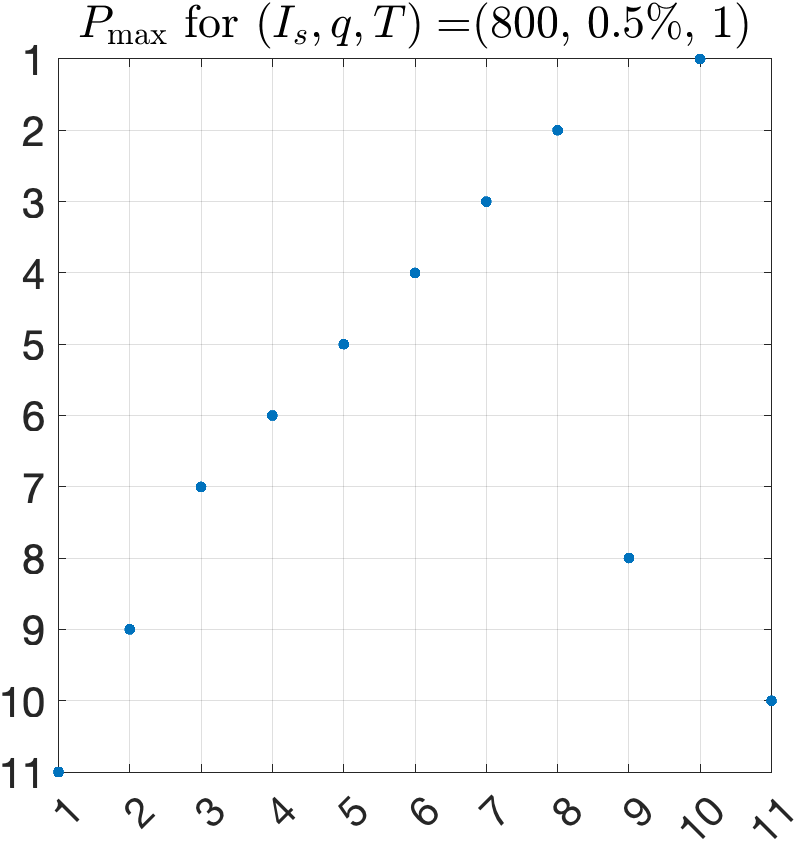}
\includegraphics[scale=0.35]{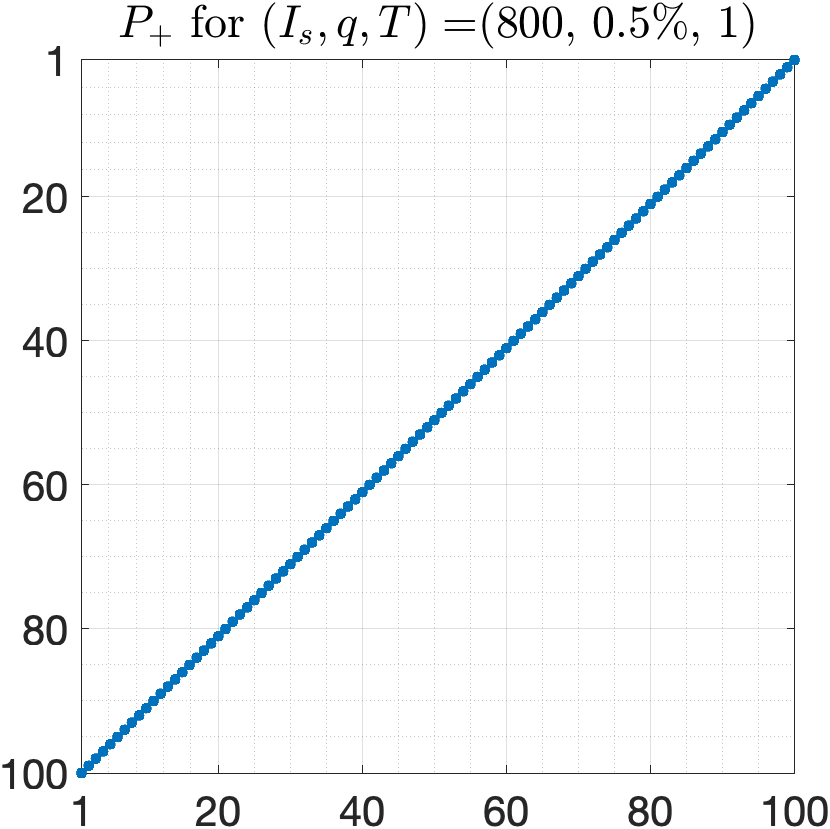}
\caption{Pattern of the optimal matrix $P_{\max}$ for Problem~\eqref{eq:Jpgen} and $N=11$ (Left) and $P_+$ for Problem~\eqref{eq:Jpapproxgen} and $N=100$ (Right) for the two parameters triplets.
The blue points represent non-zero entries, i.e., entries equal to $1$.}
\label{fig:Popt}
\end{figure}
It can be observed that for the parameter triplet (2000, 5\%, 1000), the two controls $P_+,P_{\max}$ have the same form for $N=11$ and $N=100$ (Figure~\ref{fig:Popt} Top).
Hence, one can expect $P_{\max} = P_+$ for larger $N$ which is the case until $N=11$ (as shown in Figure~\ref{fig:muN} black square).
However, this may not be the case for (800,0.5\%,1) since $P_+,P_{\max}$ have already different forms for $N=11$ (Figure~\ref{fig:Popt} Bottom).

In the following tests, we focus only on two special cases: large lap duration time ($T=\SI{1000}{s}$) and small lap duration time ($T=\SI{1}{s}$). 
In practice, the former corresponds to typical time required to complete one lap in a raceway pond system, whereas the latter rather corresponds to photobioreactor~\cite{Lamare2019}. 
In the small lap duration time case, we observe the so-called \textit{flashing effect}.
This phenomenon corresponds to the fact that the growth rate is an increasing function of the light exposition frequency.
It can be observed in Figure~\ref{fig:4muT}, where $\bar \mu_N(P_{\max})$ decreases with respect to $T$ for all considered light intensities.
This phenomenon has already been reported in literature, see, e.g.~\cite{Lamare2019}. 
\begin{figure}[htpb]
\centering
\includegraphics[scale=0.4]{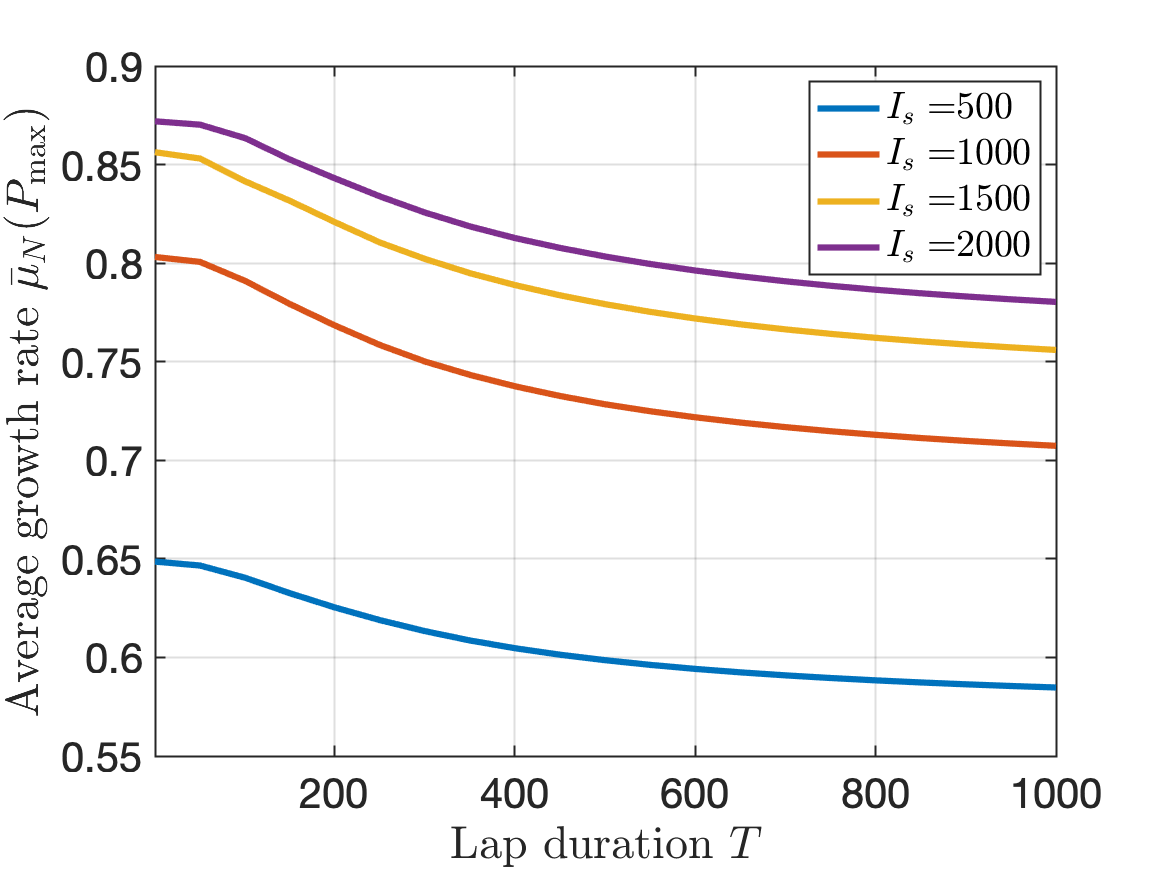}
\caption{Average specific growth rate in the case $q=0.1\%$ and $N=7$ for four different light intensities $I_s$. }
\label{fig:4muT}
\end{figure}

The next test is dedicated to the efficiency of the criterion~\eqref{eq:thmcon}.
More precisely, we evaluate the function $\bar \mu_N$ defined by~\eqref{eq:muN} for the optimal control $P_{\max}$ which solves Problem~\eqref{eq:Jpgen} and for the control $P_+$ which solves the approximated Problem~\eqref{eq:Jpapproxgen}.
We consider two different discretisation values $N = 5$ and $N=9$. 
Figure~\ref{fig:2mark} shows the results for $T=\SI{1}{s}$ and $T=\SI{1000}{s}$.
\begin{figure}[htpb]
\centering
\includegraphics[scale=0.3]{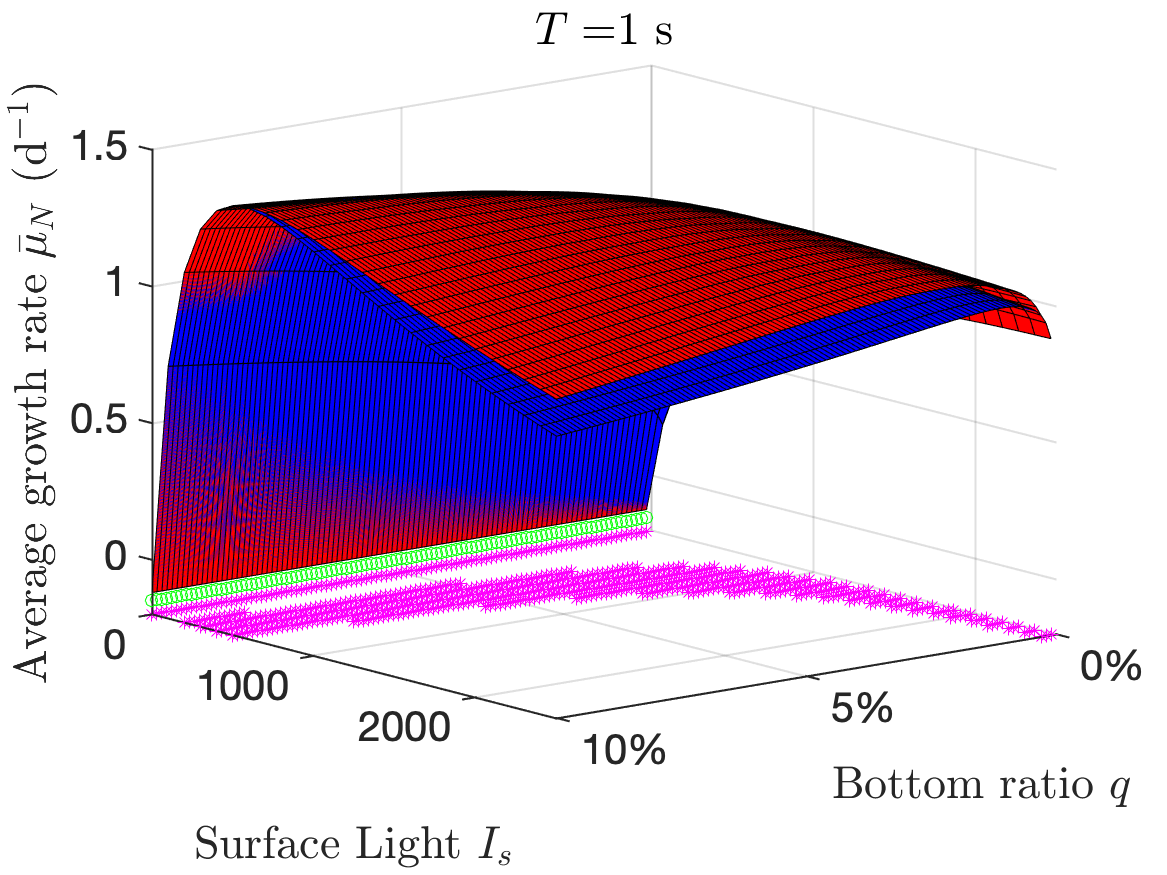}
\includegraphics[scale=0.3]{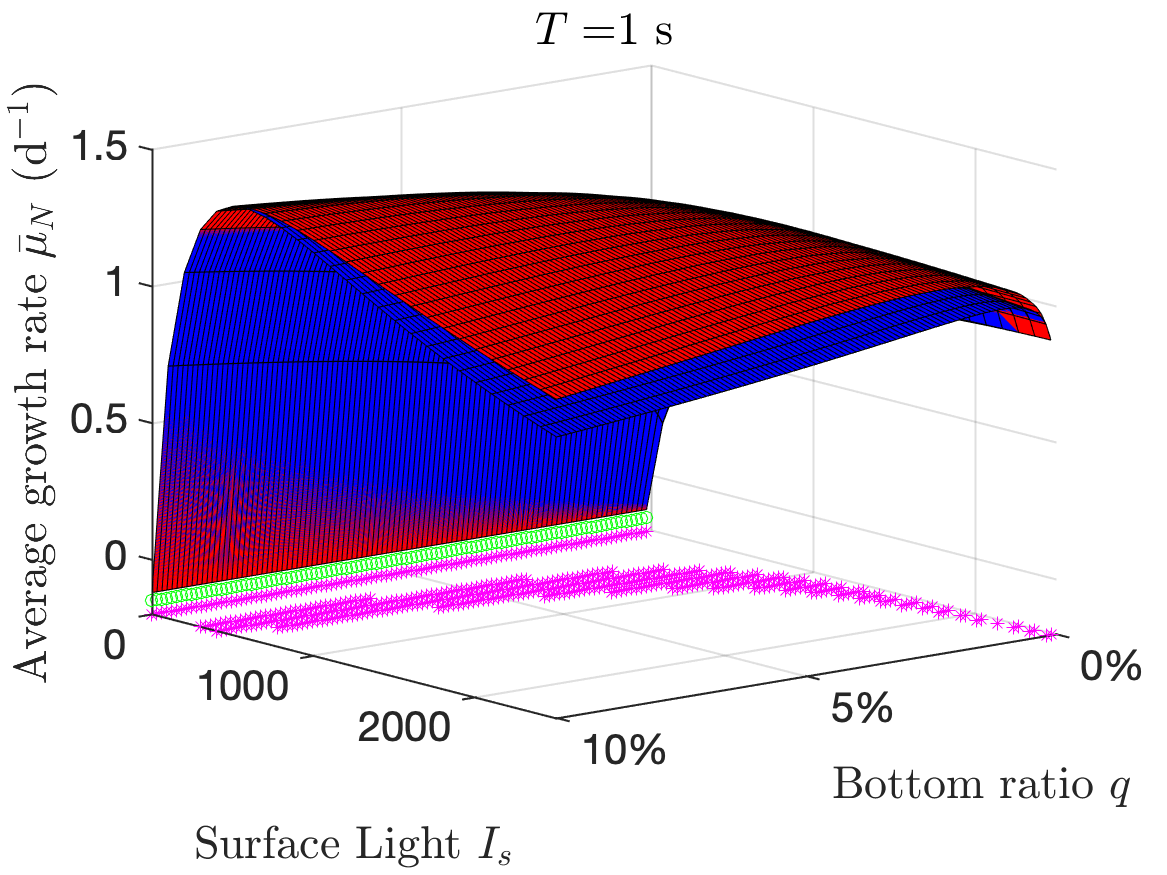}\\
\includegraphics[scale=0.3]{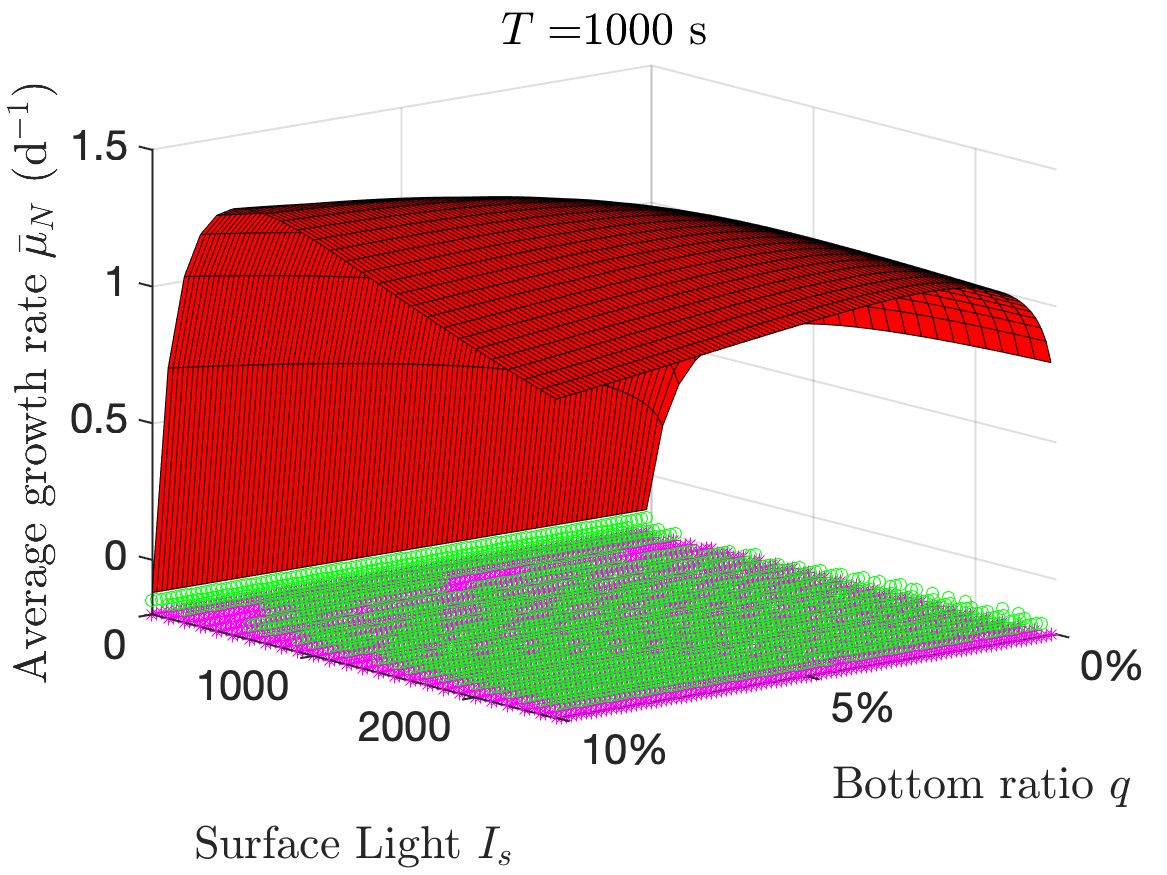}
\includegraphics[scale=0.3]{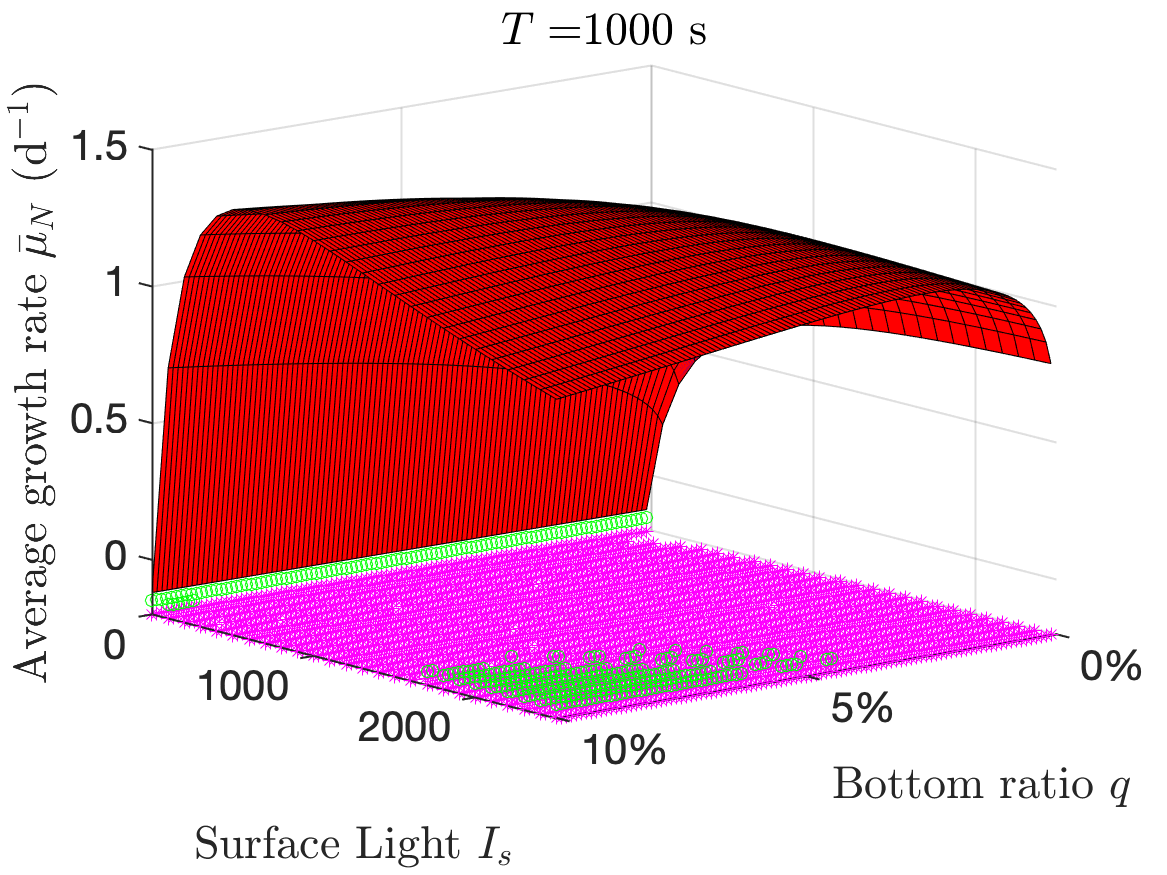}
\caption{Average net specific growth rate $\bar \mu_N$ for $T=\SI{1}{s}$ (Top) and for $T=\SI{1000}{s}$ (Bottom). 
Left: $N=5$.
Right: $N=9$.
The red surface is obtained with the control $P_{\max}$ and the blue surface is obtained with the control $P_+$.
The purple stars represent the cases where $P_{\max}=P_+$ or, in case of multiple solution, $\bar \mu_N(P_{\max})=\bar \mu_N(P_+)$. 
The green circle represent the cases where the criterion~\eqref{eq:thmcon} is satisfied.}
\label{fig:2mark}
\end{figure}
We see that for large values of $T$, the optimum approximation almost always coincides to the true optimum. 
Nevertheless, we observe that the criterion~\eqref{eq:thmcon} becomes less efficient for larger $N$.
Note that the case corresponding to $I_s=\SI{0}{\mu mol.m^{-2}.s^{-1}}$ is particular since no light is available in the system, implying that $\Gamma,V$ equal to zero.
In this case the value of the objective function do not depend on the control $P$.   
Hence $\bar \mu_N(P_{\max})=\bar \mu_N(P_+)$ when $I_s=\SI{0}{\mu mol.m^{-2}.s^{-1}}$.

We finally evaluate the efficiency of various mixing strategies. Define 
\begin{align}
&r_1:=\frac{\bar \mu_N(P_{\max})-\bar \mu_N(\I)}{\bar \mu_N(\I)}, \label{eq:r1}\\
&r_2:=\frac{\bar \mu_N(P_{\max})-\bar \mu_N(P_{\min})}{\bar \mu_N(P_{\min})},\label{eq:r2}\\
&r_3:=\frac{\bar \mu_N(\I)-\bar \mu_N(P_{\min})}{\bar \mu_N(\I)},\label{eq:r3}
\end{align}
where $P_{\min}\in\P$ is the matrix that minimizes $J$, (see~\eqref{eq:Jpgen}), i.e., that corresponds to the worse strategy.
We consider $N=9$ layers.
Figure~\ref{fig:3r} presents the results for $T=\SI{1}{s}$ and $T=\SI{1000}{s}$.
\begin{figure}[htpb]
\centering
\includegraphics[scale=0.29]{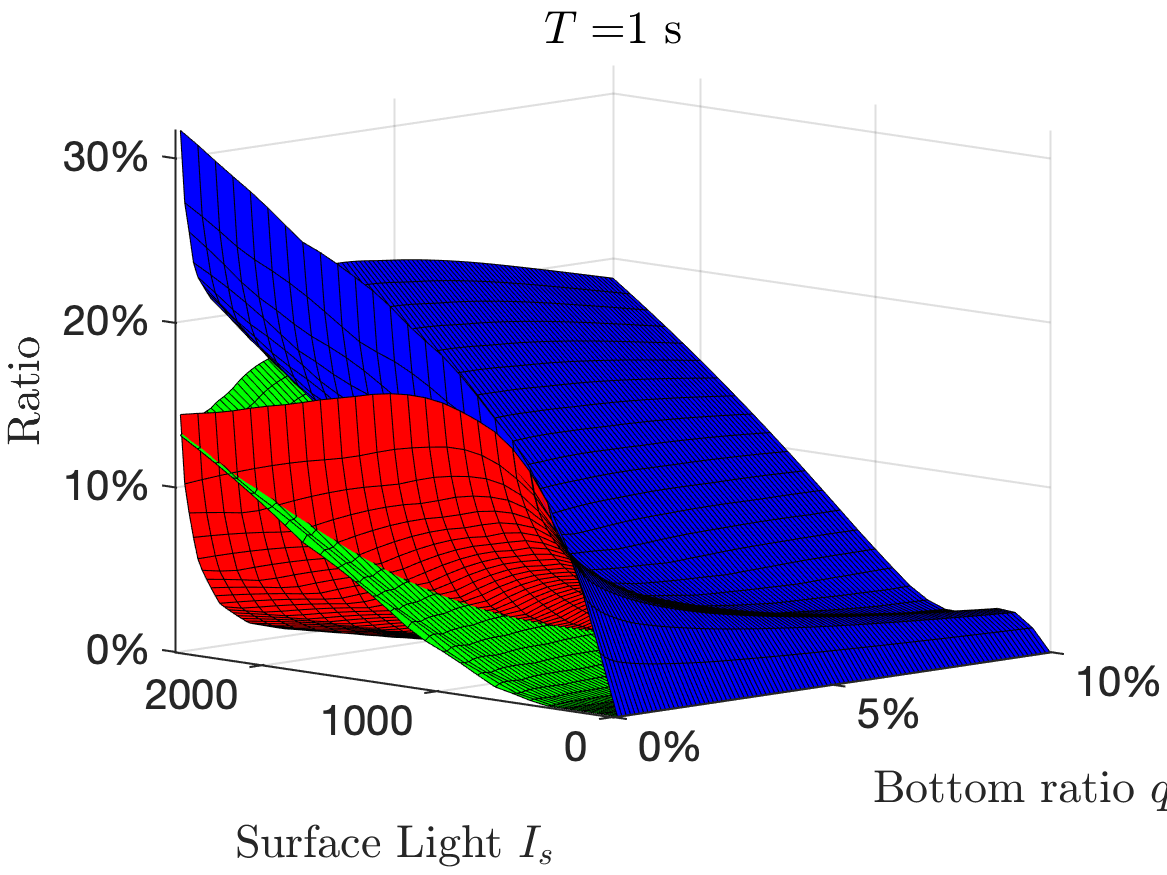}
\includegraphics[scale=0.29]{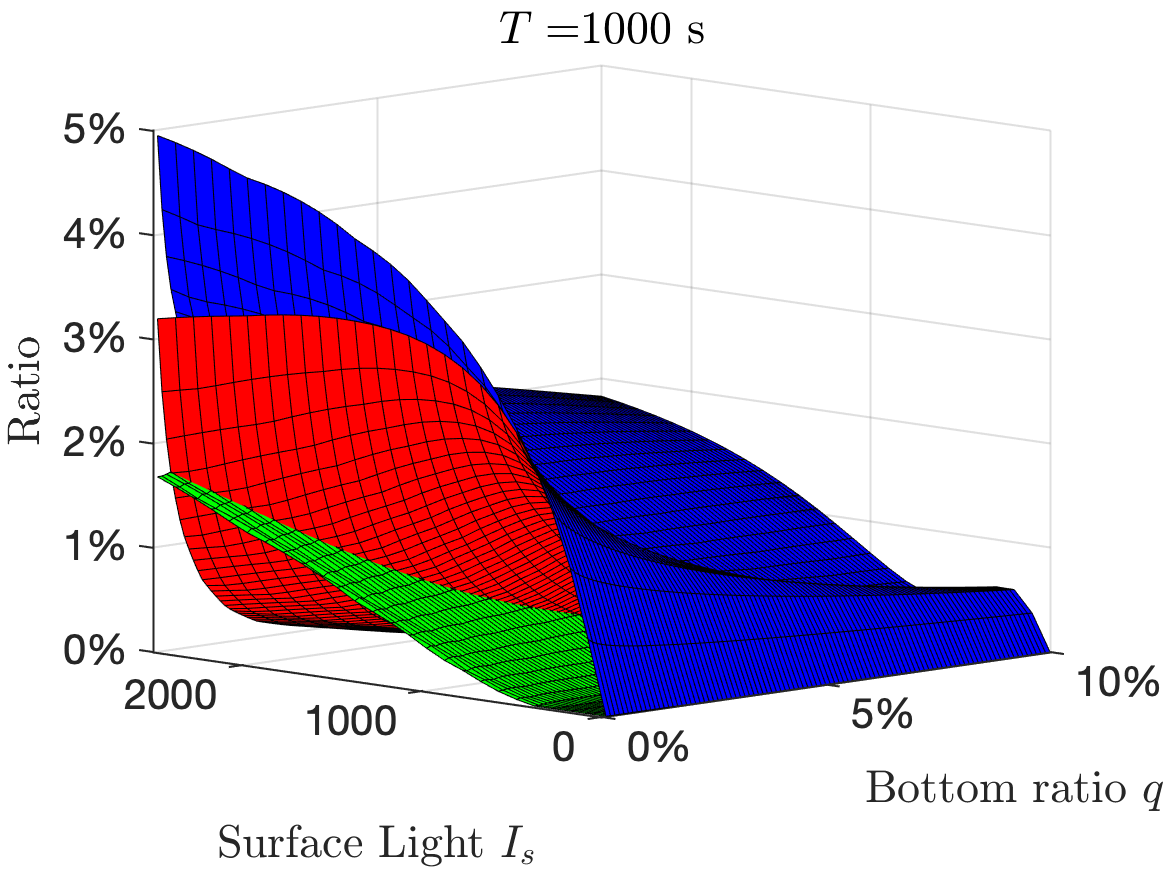}
\caption{Three ratios~\eqref{eq:r1}-~\eqref{eq:r3} for $T=\SI{1}{s}$ (Left) and for $T=\SI{1000}{s}$ (Right).
In each figure, the red surface represents $r_1$, the blue surface represents $r_2$ and the green surface represents $r_3$.}
\label{fig:3r}
\end{figure}
Better performance is in most cases obtained for a small lap duration $T=\SI{1}{s}$.
In this way, we observe that the relative improvement between the best and the no mixing strategy may reach 15\%, whereas the relative improvement between the worst and the best strategy may reach 30\%.
In both two cases, a better improvement can be obtained with high values of $I_s$ and low values of $q$.

To compare the efficiency of the approximation $P_+$ with respect the true optimal mixing strategy $P_{\max}$, we define two extra ratios:
\begin{align}
&\tilde r_1:=\frac{\bar \mu_N(P_+)-\bar \mu_N(\I)}{\bar \mu_N(\I)}, \label{eq:rt1}\\
&\tilde r_2:=\frac{\bar \mu_N(P_+)-\bar \mu_N(P_{\min})}{\bar \mu_N(P_{\min})}.\label{eq:rt2}
\end{align}
Figure~\ref{fig:2rt} presents the results for $T=\SI{1}{s}$ and $T=\SI{1000}{s}$.
\begin{figure}[htpb]
\centering
\includegraphics[scale=0.29]{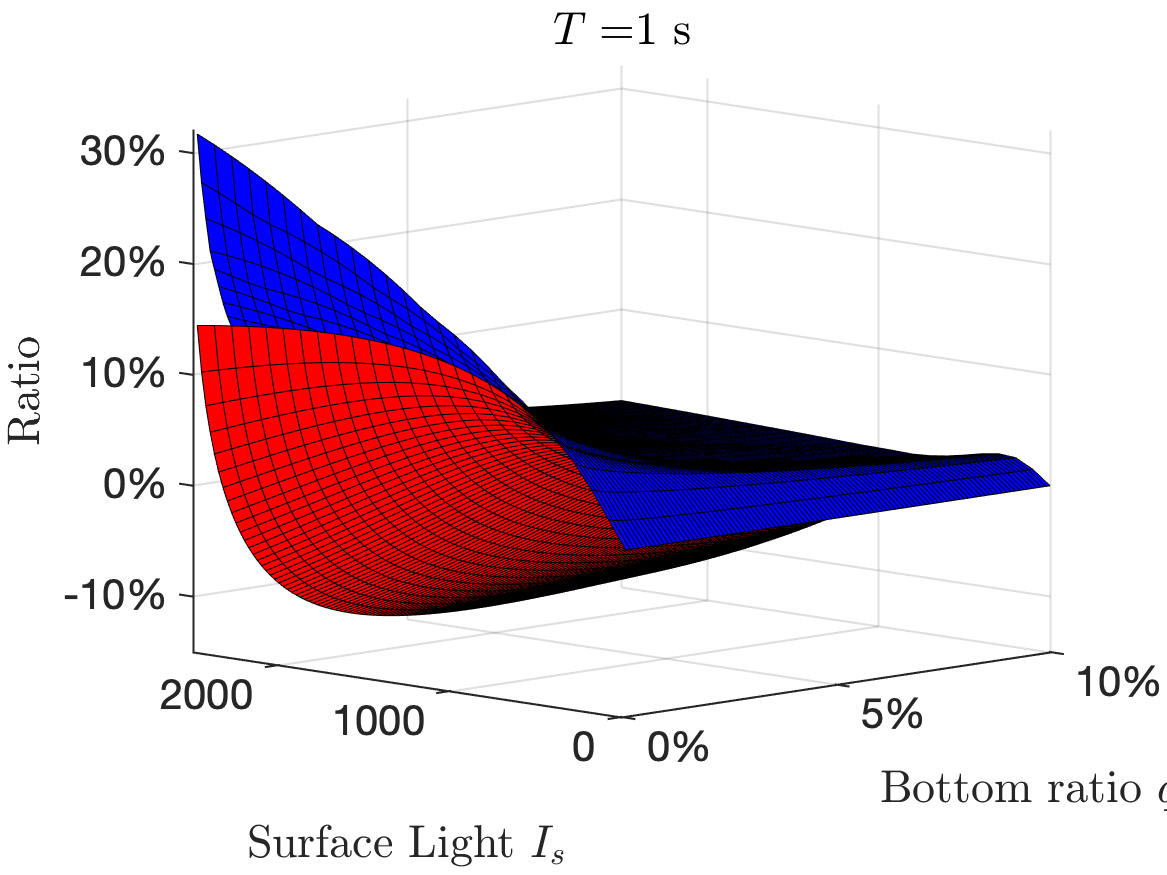}
\includegraphics[scale=0.29]{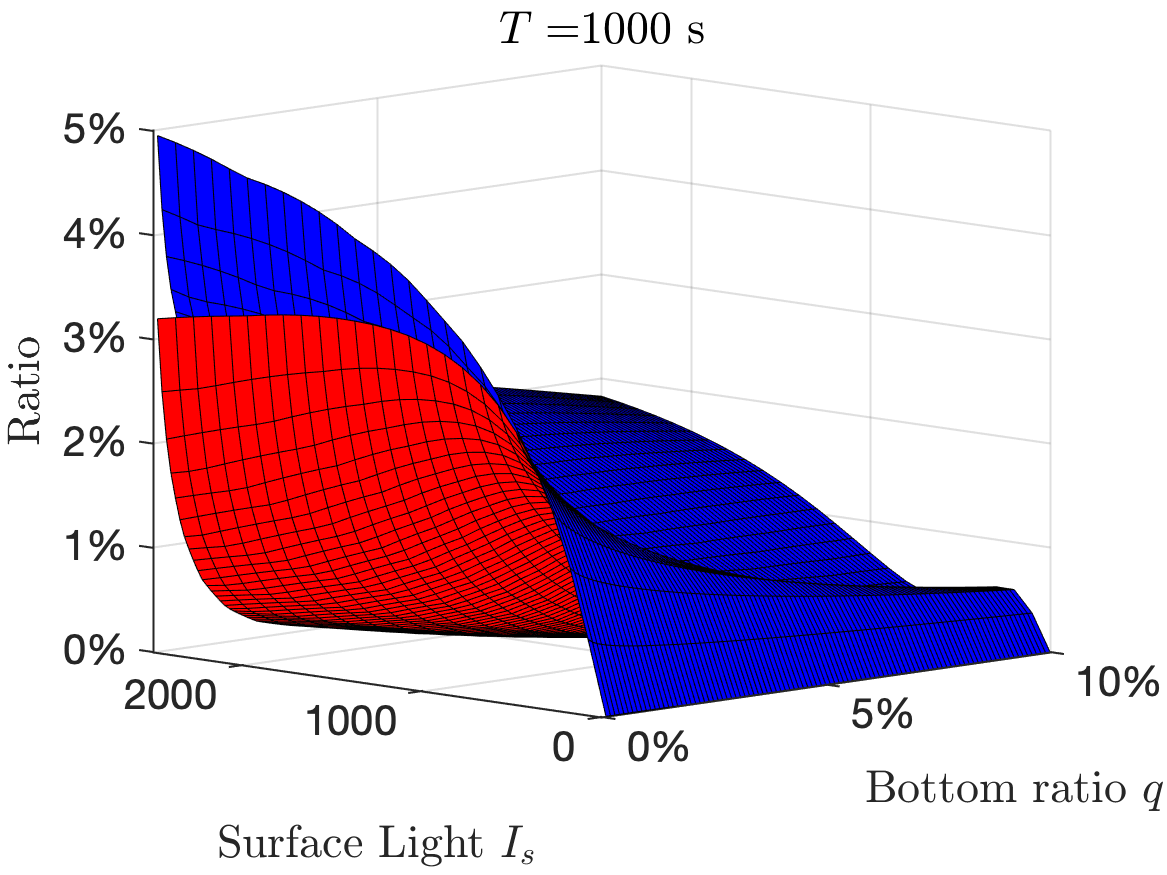}
\caption{Two ratios~\eqref{eq:rt1}-~\eqref{eq:rt2} for $T=\SI{1}{s}$ (Left) and for $T=\SI{1000}{s}$ (Right).
In each figure, the red surface represents $\tilde r_1$, the blue surface represents $\tilde r_2$.}
\label{fig:2rt}
\end{figure}
As already mentioned, for a large lap duration time, the optimization problem~\eqref{eq:Jpapproxgen} provides a good approximation.

This can be observed with the blue and red surface in Figure~\ref{fig:3r} (Right) and in Figure~\ref{fig:2rt} (Right), both surfaces have the same behaviours.
As expected, the approximation becomes less efficient in the case of short lap duration time.
This can be observed in Figure~\ref{fig:3r} (Left) and in Figure~\ref{fig:2rt} (Left). 
However, the maximal values of $r_1,r_2$ are still preserved by their approximations $\tilde r_1, \tilde r_2$.

\section{Conclusion}\label{sec:conclusion}

We have studied a periodic resource allocation problem combined with a  dynamical system.
The periodicity of the problem enables us to reduce the computation to one assignment process.
A significant computational effort is still required when dealing with larger number of $N$. 
We overcome this difficulty by defining a second optimization problem which has an explicit solution that coincide with the true solution when a given criterion is satisfied.

This developed theory is then applied to a microalgal production system with a mixing device.
Non-trivial optimal mixing strategies can be obtained and the proposed second optimization problem provides a reliable approximation for large time duration $T$.
Besides, our experimental results show the significance of the choice of the mixing strategy: the relative ratio between the best and the worst case reaches 30\% in some cases. 
We also observe a flashing effect meaning that better results are obtained when $T$ goes to zero.

Further works will be devoted to the improvement of the function $\phi$ used in Theorem~\ref{thm:main} in order to improve our approach for large number of $N$.
An approximation problem for small lap duration $T$ can also be considered with an appropriate criterion to evaluate this approximation.

\bibliography{perm}
\bibliographystyle{plain}

\appendix

\section{Explicit Computations}\label{app:exacomp}

In this appendix, we provide the computational details to solve~\eqref{eq:dotC} and~\eqref{eq:muN} for an arbitrary number $n\in\{1,\ldots, N\}$.
Given two points $t_1, t_2\in[0,T]$. 
Since $I_n$ is constant, Equation~\eqref{eq:evol} can be integrated and becomes
\begin{equation}\label{eq:cnt}
C_n(t_2) = e^{\alpha(I_n)(t_1-t_2)}C_n(t_1) + \frac{\beta(I_n)}{\alpha(I_n)}(1 - e^{\alpha(I_n)(t_1-t_2)}).
\end{equation}
The time integral in~\eqref{eq:muN} can be computed by
\begin{equation*}
\begin{split}
\int_0^T \mu(C_n(t),I_n) \D t= &\int_0^T -\gamma(I_n)C_n(t) + \zeta(I_n) \D t\\
                            &= -\gamma(I_n) \int_0^TC_n(t) \D t + \zeta(I_n) T.
\end{split}
\end{equation*}
Replacing $x_n$ by $C_n$, $t_2$ by $t$ and $t_1$ by 0 in~\eqref{eq:cnt} and integrating $t$ from $0$ to $T$ gives
\begin{equation*}
\begin{split}
\int_0^TC_n(t) \D t =&\int_0^T \Big(e^{-\alpha(I_n)t}C_n(0) + \frac{\beta(I_n)}{\alpha(I_n)}(1 - e^{-\alpha(I_n)t}) \Big)\D t\\
=&\frac{C_n(0)}{\alpha(I_n)}(1-e^{-\alpha(I_n)T}) +  \frac{\beta(I_n)}{\alpha(I_n)}T - \frac{\beta(I_n)}{\alpha^2(I_n)}(1-e^{-\alpha(I_n)T}).
\end{split}
\end{equation*}
Using notations given in Section~\ref{sec:example}, we have 
\begin{equation*}
\Gamma = \frac{\gamma(I)}{\alpha(I)}(e^{-\alpha(I)T}-1), \quad V = \frac{\beta(I)}{\alpha(I)}(1 - e^{-\alpha(I)T}).
\end{equation*}
From the definition of $\alpha(I),\beta(I),\gamma(I)$, we find
\begin{equation*}
\begin{split}
\frac{\beta(I)}{\alpha(I)} &= \frac{\beta(I)}{\beta(I)+k_r} = \frac{k_d\tau(\sigma_H I)^2}{k_d\tau(\sigma_H I)^2 + k_r\tau\sigma_H I + k_r},\\
\frac{\gamma(I)}{\alpha(I)} &= \frac{k_H\sigma_H I}{k_d\tau(\sigma_H I)^2 + k_r\tau\sigma_H I + k_r}.
\end{split}
\end{equation*}
Remark that $\Gamma$ and $V$ always have the opposite sign.
Note also that $I\mapsto\frac{\beta(I)}{\alpha(I)}$ is increasing on $[0,+\infty)$, which is not the case for $I\mapsto\frac{\gamma(I)}{\alpha(I)}$.
It follows that $V$ increases on $\R^+$ and $\Gamma$ is not monotonic on $\R^+$ (see Figure~\ref{fig:gammaV}).

\section{Optimization problem with arbitrary vectors}\label{app:usort}

Let $\tilde u, v\in\R^N$ two arbitrary vectors.
Let $Q\in\P$ such that $u := Q \tilde u$ has entries sorted in ascending order.
Since $Q$ is a permutation matrix, we have $Q^T = Q^{-1}$.
For any $P\in\P$, let us denote by $\tilde P := Q^{-1}P Q$, we have $\tilde P\in\P$ a permutation matrix.
Let us denote by $\tilde v := Q^{-1} v$ and by $\tilde D = Q^{-1}D Q$. 
Note that $\tilde D$ is still a diagonal matrix with a different order of the diagonal coefficients. 
Using this notation, we find for the objective function~\eqref{eq:Jpgen} satisfies
\begin{equation*}
\begin{split}
J(P):=\langle u, (\I - PD)^{-1}P v\rangle &= \langle \tilde u, Q^{-1} (\I - PD)^{-1}Q Q^{-1}P Q Q^{-1}v\rangle \\
&=\langle \tilde u, \big(Q^{-1}(\I - PD)Q\big)^{-1} \tilde P\tilde v\rangle\\
&=\langle \tilde u, (Q^{-1}Q - Q^{-1}P Q Q^{-1}D Q)^{-1} \tilde P\tilde v\rangle\\
&=\langle \tilde u, (\I - \tilde P  \tilde D)^{-1} \tilde P\tilde v\rangle.
\end{split}
\end{equation*}
For  the objective function~\eqref{eq:Jpapproxgen}, we get
\begin{equation*}
J^{\text{approx}}(P):=\langle u, P v\rangle = \langle \tilde u, Q^{-1}P Q Q^{-1} v\rangle =  \langle \tilde u, \tilde P \tilde v\rangle.
\end{equation*}
Therefore, these problems can still be treated similarly in the general case.

\section{Remark on \texorpdfstring{$F_m^+,F_m^-$}{Lg}}\label{app:Fm}

Let $u, v\in\R^N$ such that the entries of $u$ are sorted in ascending order. 
One should be careful when defining the two sequences $F^+_m$ and $F^-_m$ in Section~\ref{sec:criterion}, since the sign of $u$ and $v$ plays an important role in the definition of these two sequences. 
For instance, assume that $u$ is now negative and $v$ is positive.
Let $\tilde u:=-u$, since $u$ is assumed to be sorted in ascending order, $\tilde u$ is positive and sorted in descending order.
Using the definition in~\eqref{eq:Fm}, one has
\begin{equation*}
\tilde F^+_m := \sum_{n=1}^{\min(m,N)} \tilde u_n v_{\tilde \sigma_+(n)}, \quad \tilde F^-_m := \sum_{n = \max(1,N-m+1)}^N \tilde u_n v_{\tilde \sigma_-(2N-m-n+1)},
\end{equation*}
where $v_{\tilde \sigma_+(1)} \geq v_{\tilde \sigma_+(2)} \geq ,\ldots, \geq v_{\tilde \sigma_+(N)}$ and $v_{\tilde \sigma_-(1)} \leq v_{\tilde \sigma_-(2)} \leq ,\ldots, \leq v_{\tilde \sigma_-(N)}$.
Let us define by $\sigma_+ := \tilde \sigma_-$ and $\sigma_- := \tilde \sigma_+$.
One has 
\begin{equation*}
\tilde F^+_m = -\sum_{n=1}^{\min(m,N)} u_n v_{\sigma_-(n)}, \quad \tilde F^-_m = -\sum_{n = \max(1,N-m+1)}^N u_n v_{\sigma_+(2N-m-n+1)}.
\end{equation*}
Therefore, in this case we can define $F^+_m$ and $F^-_m$ by
\begin{equation*}
F^-_m := \sum_{n=1}^{\min(m,N)} u_n v_{\sigma_-(n)}, \quad F^+_m := \sum_{n = \max(1,N-m+1)}^N u_n v_{\sigma_+(2N-m-n+1)}.
\end{equation*}
The case where $u$ is positive and $v$ is negative, or both $u,v$ are negative can be treated similarly. 

\end{document}